\documentclass[11pt,reqno]{amsart}
\usepackage[utf8]{inputenc}

\usepackage{amssymb,latexsym,amsmath,amsthm,bbm,caption}
\usepackage{graphicx,multicol,enumitem}
\usepackage{fullpage}
\usepackage{tabstackengine}
\stackMath

\numberwithin{equation}{section}
\newcounter{Theorem}
\numberwithin{Theorem}{section}
\newtheorem{theorem}[Theorem]{Theorem}
\newtheorem{lemma}[Theorem]{Lemma}
\newtheorem{proposition}[Theorem]{Proposition}
\newtheorem{corollary}[Theorem]{Corollary}
\theoremstyle{definition}

\numberwithin{Definition}{section}

\DeclareMathOperator*{\half}{\textstyle{\frac{1}{2}}}

\allowdisplaybreaks

\begin{document}

\title{Stability of iterated dyadic filter banks}
\author{Marcin Bownik}
\address{Department of Mathematics, University of Oregon, Eugene, OR 97403-1222}
\email{mbownik@uoregon.edu}
\author{Brody Johnson}
\address{Department of Mathematics and Statistics, Saint Louis University, St. Louis, MO 63103}
\email{brody.johnson@slu.edu}
\author{Simon McCreary-Ellis}
\address{Department of Mathematical Sciences, United States Air Force Academy, United States Air Force Academy, CO 80840}
\email{Simon.McCreary-Ellis@afacademy.af.edu}

\begin{abstract}
This paper examines the frame properties of finitely and infinitely iterated dyadic filter banks.  It is shown that the stability of an infinitely iterated dyadic filter bank guarantees that of any associated finitely iterated dyadic filter bank with uniform bounds.  Conditions under which the stability of finitely iterated dyadic filter banks with uniform bounds implies that of the infinitely iterated dyadic filter bank are also given.  The main result describes a sufficient condition under which the infinitely iterated dyadic filter bank associated with a specific class of finitely supported filters is stable.
\end{abstract}

\thanks{The first author was partially supported by the NSF grant DMS-1956395.  The authors wish to thank Jarek Kwapisz for sharing an alternative, unpublished proof for Theorem \ref{expandingFB} as well as for helpful conversations that led to improvements in this paper.} 



\maketitle

\section{Introduction} \label{inro}

This paper is concerned with the frame properties of infinitely iterated dyadic filter banks, as illustrated in Figure \ref{IFBanalysis}.  The advent of the multiresolution analysis (MRA) in the late 1980s \cite{Mallat1989} provided a fruitful connection between such filter banks and the theory of dyadic orthonormal wavelets on the line, which, depending on one's goals, may be exploited in either direction.  In particular, in one direction, it was realized that the masks appearing in the refinement equations of the scaling function and wavelet naturally give rise to a class of perfect reconstruction filter banks \cite[\S 5.6]{Daubechies1992}.  One advantage of such perfect reconstruction filter banks stems from the fact that they implement an orthonormal decomposition of the original sequence and thus remain stable under arbitrarily many iterations.  Characterizations of low-pass filters associated with MRA wavelets were first obtained for trigonometric polynomial filters by Lawton \cite{Lawton1990,Lawton1991} and Cohen \cite{Cohen1990}, eventually culminating with the complete characterization of low-pass filters by Gundy \cite{Gundy2000}.  Characterizations of low-pass filters associated with Parseval frame wavelets were also studied \cite{Lawton1990,Lawton1991,PSW1999,PSWX2001}.

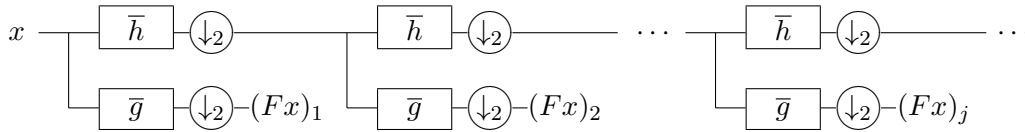
\begin{figure}[hbtp]
\begin{center}
  \begin{picture}(140,15)(1,35)
	\put (1,45){$x$}
	\put (5,46){\line(1,0){8}}

	\put (9,46){\line(0,-1){10}}
    \put (13,45){\framebox[10mm]{$\overline{h}$}}
	\put (23,46){\line(1,0){2}}
	\put (28,46){\circle{6}}
	\put (25,45){\makebox[6mm]{$\downarrow_{2}$}}
	\put (31,46){\line(1,0){19}}
	\put (9,36){\line(1,0){4}}
    \put (13,35){\framebox[10mm]{$ \overline{g}$}}
	\put (23,36){\line(1,0){2}}
	\put (28,36){\circle{6}}
	\put (25,35){\makebox[6mm]{$\downarrow_{2}$}}
	\put (31,36){\line(1,0){2}}
	\put (34,35){\makebox[8mm]{$(Fx)_{1}$}}

	\put (46,46){\line(0,-1){10}}
    \put (50,45){\framebox[10mm]{$\overline{h}$}}
	\put (65,46){\circle{6}}
	\put (62,45){\makebox[6mm]{$\downarrow_{2}$}}
	\put (68,46){\line(1,0){14}}
	\put (60,46){\line(1,0){2}}
	\put (46,36){\line(1,0){4}}
    \put (50,35){\framebox[10mm]{$ \overline{g}$}}
	\put (60,36){\line(1,0){2}}
	\put (65,36){\circle{6}}
	\put (62,35){\makebox[6mm]{$\downarrow_{2}$}}
	\put (68,36){\line(1,0){2}}
	\put (71,35){\makebox[8mm]{$(Fx)_{2}$}}

    \put (83,45){\makebox[8mm]{$\cdots$}}

	\put (91,46){\line(1,0){8}}
	\put (95,46){\line(0,-1){10}}
    \put (99,45){\framebox[10mm]{$\overline{h}$}}
	\put (114,46){\circle{6}}
	\put (111,45){\makebox[6mm]{$\downarrow_{2}$}}
	\put (117,46){\line(1,0){14}}
	\put (109,46){\line(1,0){2}}
	\put (95,36){\line(1,0){4}}
    \put (99,35){\framebox[10mm]{$ \overline{g}$}}
	\put (109,36){\line(1,0){2}}
	\put (114,36){\circle{6}}
	\put (111,35){\makebox[6mm]{$\downarrow_{2}$}}
	\put (117,36){\line(1,0){2}}
	\put (120,35){\makebox[8mm]{$(Fx)_{j}$}}
    \put (131,45){\makebox[8mm]{$\cdots$}}

 \end{picture}
\end{center} 

\caption{Analysis schematic for an infinitely iterated dyadic filter bank.} \label{IFBanalysis}
\end{figure}

Interest quickly developed in a broader class of wavelets and a correspondingly broader class of low- and high-pass filters which might possess useful properties found to be incompatible with orthonormal wavelets.  This interest motivated the construction of biorthogonal wavelets by Cohen, Daubechies, and Feauveau \cite{CohenDaubechiesFeauveau1992} in which two pairs of low- and high-pass filters give rise to a pair of iterated filter banks that are dual to one another and stable under arbitrarily many iterations.  As in the orthonormal case, the perfect reconstruction property shared by the dual pairs of filter banks is critical to the proof of their stability under iteration.  Interestingly, it was discovered that an early class of low-pass filters studied by Burt and Adelson \cite{BurtAdelson1983} were compatible with the construction of biorthogonal wavelets described by Cohen, Daubechies, and Feauveau \cite[\S 6.C.1]{CohenDaubechiesFeauveau1992}.

More recently, Bayram and Selesnick \cite{BayramSelesnick2009} investigated the relationship between the frame properties of dyadic MRA wavelets on the line and the associated iterated filter banks.  In particular, they proved that if the dyadic wavelet system constitutes a Riesz basis for $L^{2}(\mathbb{R})$ with bounds $A$, $B$, then the finitely iterated dyadic filter bank is a Riesz basis with bounds $A/B$, $B/A$ for any number of iterations \cite[Theorem 1]{BayramSelesnick2009}.  Conversely, they showed that if the finitely iterated dyadic filter bank is a frame with bounds independent of the number of iterations, then the wavelet system must be a frame for $L^{2}(\mathbb{R})$ \cite[Theorem 7]{BayramSelesnick2009}.  In the Discussion section of their paper \cite{BayramSelesnick2009}, Bayram and Selesnick went on to ask:
\begin{quote}
\emph{...what are the conditions, if any, on the filters (directly, that is, without referring to the scaling function or the wavelet) which will yield a non-perfect reconstruction system (but will possibly possess other useful properties) and will be stable under iterations?}
\end{quote}

\noindent
In the dyadic case, one answer to this question had already been provided, indirectly, by Han \cite[Theorem 6]{Han2005}, who demonstrated [separate] necessary and sufficient conditions on the low- and high-pass filters for the wavelet $\psi$ to give rise to a Riesz basis of $L^{2}(\mathbb{R})$.  In light of Bayram and Selesnick's findings, these characterizations also apply to the corresponding iterated dyadic filter banks.  Han's approach focuses on a quantity $\nu_{2}$ that is related to both the convergence of the cascade algorithm as well as the smoothness of the associated scaling function $\varphi$.  Han and Jia \cite[Theorem 1.1]{HanJia2007} subsequently derived a complete characterization of Riesz bases arising from compactly supported MRA wavelets in higher dimensions in terms of the spectral radius of a certain transition operator.  Although Han and Jia do not explicitly discuss it, their characterization also extends to the stability of the associated iterated filter banks.

Interest in non-perfect reconstruction filter banks has increased with the widespread use of the discrete wavelet transform in machine learning applications such as feature extraction or pattern recognition \cite{Balan2018, Bolcskei2019, WiatowskiBolcskei2018}.  In such contexts, reconstruction is seldom required, while the stability of the filter bank is needed to guarantee the uniqueness of representation.  Hence, the present work seeks to further address the question posed by Bayram and Selesnick by augmenting the understanding of iterated filter bank frames in two directions.  First, the relationship between the stability of an infinitely iterated dyadic filter bank and that of the corresponding finitely iterated dyadic filter banks is examined.  Theorem \ref{IIFBtoFIFB} shows that stability of an infinitely iterated dyadic filter bank guarantees the stability of the corresponding finitely iterated dyadic filter banks with uniform bounds.  Conversely, Theorem \ref{FIFBtoIIFB} shows that the stability of finitely iterated dyadic filter banks with uniform bounds implies the stability of the associated infinitely iterated dyadic filter bank provided that the norm of the component due to the low-pass filter tends to zero as the number of iterations tends to infinity.  Second, easily verifiable sufficient conditions on the low- and high-pass filters are described that guarantee the stability of an infinitely iterated dyadic filter bank and, consequently, the stability of any finitely iterated dyadic filter bank with uniform bounds.  Theorem \ref{expandingFB} describes a sufficient condition for a broad class of finitely supported filters that guarantees the stability of the associated infinitely iterated dyadic filter bank.  In light of Theorem \ref{IIFBtoFIFB}, these filters also give rise to finitely iterated dyadic filter banks which are stable for any number of iterations and with uniform bounds.  Thus, the main contributions of this work are simple sufficient conditions on the filters guaranteeing the stability of an infinitely iterated dyadic filter bank and a result showing that the stability of the infinitely iterated dyadic filter bank implies that of any finitely iterated dyadic filter bank associated with the same filters (with uniform bounds).

\section{Preliminaries} \label{preliminaries}

This section develops terminology and notation that will be used throughout the remaining sections.   A \emph{frame} for a separable Hilbert space $\mathbb{H}$ is a collection $\lbrace e_{j} \rbrace_{j\in J} \subset \mathbb{H}$, where $J$ is a countable index set, for which there exist constants $0<A\le B<\infty$ (called \emph{frame bounds}) such that for all $x\in \mathbb{H}$,
$$ A \Vert x\Vert_{\mathbb{H}}^{2} \le \sum_{j\in J} \vert \langle x, e_{j}\rangle \vert^{2} \le B \Vert x\Vert_{\mathbb{H}}^{2}.$$

\noindent
When it is possible to choose $A=B$, the frame is said to be \emph{tight}. A \emph{Parseval} frame is a tight frame for which $A=B=1$.  If only the right-hand inequality holds, the collection $\lbrace e_{j} \rbrace_{j\in J}$ is called a \emph{Bessel system} and $B$ is called the Bessel bound.  

For $x\in \ell^{2}(\mathbb{Z})$ the \emph{Fourier transform of $x$} will be denoted by $\hat{x}$ and is defined as the $1$-periodic function on $\mathbb{R}$ given by
\begin{equation*}
\hat{x}(\xi) = \sum_{n\in\mathbb{Z}} x(n) e^{-2\pi i n \xi}, \quad \xi \in \mathbb{R}.
\end{equation*}

\noindent
In many cases identities involving such Fourier transforms will be considered on $\mathbb{T}$, identified here with the interval $\lbrack -\frac{1}{2}, \frac{1}{2} )$.  The convolution of sequences $x,h\in \ell^{2}(\mathbb{Z})$ is defined by
$$ (x*h)(k) = \sum_{n\in \mathbb{Z}} h(n) x(k-n), \quad k\in \mathbb{Z},$$

\noindent
which, under the Fourier transform, corresponds to $\widehat{x*h}(\xi) = \hat{x}(\xi) \hat{h}(\xi)$.  The involution of a sequence $x$ will be denoted by $\overline{x}$ and is defined by $\overline{x}(k) = \overline{x(-k)}$, so that $\widehat{\overline{x}}(\xi) = \overline{\hat{x}(\xi)}$.

The term \emph{filter} will refer to a sequence in $\ell^{2}(\mathbb{Z})$ that acts on a \emph{signal} in $\ell^{2}(\mathbb{Z})$ by convolution.  A generic signal will frequently be denoted by $x$.  The letter $h$ will be used exclusively to represent \emph{low-pass} filters, which are assumed to satisfy $\hat{h}(0)=\sqrt{2}$ and $\hat{h}(\half)=0$.  Similarly, the letter $g$ will be reserved for \emph{high-pass} filters, which are assumed to satisfy $\hat{g}(0)=0$.  Any additional assumptions on the low- or high-pass filters will be clearly stated in the hypotheses of the corresponding theorem.  In the context of orthonormal wavelets, the high-pass filter is typically constructed from the low-pass filter by
\begin{equation} \label{stdHP}
\hat{g}(\xi) = e^{-2\pi i \xi} \hat{h} (\xi + \half),
\end{equation}

\noindent
so that $\hat{g}(0)=0$ and $\hat{g}(\half)=\sqrt{2}$.  This choice of high-pass filter will be referred to as the \emph{orthogonal} high-pass filter associated with a given low-pass filter $h$.  This assumption will not be used in any of the theorems in this work, although it will be convenient for examples.

Referring to Figure \ref{IFBanalysis}, the rectangular boxes in the block diagram represent convolution with the filter indicated within.  The symbol $\downarrow_{2}$ respresents the \emph{downsampling} operator $D:\ell^{2}(\mathbb{Z})\rightarrow \ell^{2}(\mathbb{Z})$, defined by
$$ D x(n) = x(2n).$$

\noindent
The adjoint of $D$ is the \emph{upsampling} operator $U:\ell^{2}(\mathbb{Z})\rightarrow \ell^{2}(\mathbb{Z})$ given by
$$U x(n) = \begin{cases} x(m) & n=2m, \\ 0 & \text{otherwise}. \end{cases}$$

\noindent
Under the Fourier transform, these operators correspond to periodization and dilation, and for $j\in \mathbb{N}$ obey the identities
\begin{equation*}
\widehat{D^{j} x}(\xi) = 2^{-j} \sum_{k=0}^{2^{j}-1} \hat{x} ( 2^{-j}(\xi+k)) \qquad \text{and} \qquad \widehat{U^{j} x}(\xi) = \hat{x} (2^{j} \xi).
\end{equation*}

\noindent
The \emph{filter bank analysis operator} defined by the filters $h, g \in \ell^{2}(\mathbb{Z})$ and acting on $x\in \ell^{2}(\mathbb{Z})$ is the mapping $F:\ell^{2}(\mathbb{Z}) \rightarrow \bigoplus_{j=1}^{\infty} \ell^{2}(\mathbb{Z})$ defined by
\begin{equation*}
F: x \mapsto Fx := \lbrace (Fx)_{j} \rbrace_{j=1}^{\infty}.
\end{equation*}

\noindent
The block diagram of Figure \ref{IFBanalysis} can be restructured to better illustrate the origin of the components $(Fx)_{j}$, $j\in \mathbb{N}$.  Observe that
$$ (Dx*h)(k) = \sum_{n\in \mathbb{Z}} h(n) x(2k-2n) = \sum_{n\in \mathbb{Z}} (Uh)(2n) x(2k-2n) = \sum_{n\in \mathbb{N}} (Uh)(n) x(2k-n) = (D(x*Uh))(k),$$

\noindent
relating the action of downsampling and upsampling through convolution.  This relationship, one of the two Noble Identities found in the engineering literature \cite{Vaidyanathan1995}, allows the order of the downsampling and convolution operations to be reversed, leading to the alternative block diagram for the infinitely iterated filter bank shown in Figure \ref{IFBanalysis2}.

\begin{figure}[hbtp]
\begin{center}
  \begin{picture}(115,35)(1,15)
	\put (1,45){$x$}
	\put (5,46){\line(1,0){8}}

	\put (9,46){\line(0,-1){10}}
    \put (13,45){\framebox[14mm]{$\overline{g}$}}
	\put (27,46){\line(1,0){68}}
	\put (98,46){\circle{6}}
	\put (95,45){\makebox[6mm]{$\downarrow_{2}$}}
	\put (101,46){\line(1,0){4}}
	\put (106,45){\makebox[8mm]{$(Fx)_{1}$}}

	\put (9,36){\line(1,0){4}}
    \put (13,35){\framebox[14mm]{$ \overline{h}$}}
	\put (27,36){\line(1,0){4}}
    \put (31,35){\framebox[14mm]{$\overline{Ug}$}}
	\put (45,36){\line(1,0){50}}
	\put (98,36){\circle{6}}
	\put (95,35){\makebox[6mm]{$\downarrow_{4}$}}
	\put (101,36){\line(1,0){4}}
	\put (106,35){\makebox[8mm]{$(Fx)_{2}$}}

	\put (9,36){\line(0,-1){4}}
    \put (8.5,28){$\vdots$}
    \put (9,27){\line(0,-1){4}}
    
	\put (9,23){\line(1,0){4}}
    \put (13,22){\framebox[14mm]{$ \overline{h}$}}
	\put (27,23){\line(1,0){4}}
    \put (31,22){\framebox[14mm]{$ \overline{Uh}$}}
	\put (45,23){\line(1,0){4}}
	\put (50,22){$\cdots$}
	\put (55,23){\line(1,0){4}}
    \put (59,22){\framebox[14mm]{$ \overline{U^{j-2}h}$}}
	\put (73,23){\line(1,0){4}}
    \put (77,22){\framebox[14mm]{$ \overline{U^{j-1}g}$}}
	\put (91,23){\line(1,0){4}}
	\put (98,23){\circle{6}}
	\put (95,22){\makebox[6mm]{$\downarrow_{2^{j}}$}}
	\put (108,28){$\vdots$}
	\put (101,23){\line(1,0){4}}
	\put (106,22){\makebox[8mm]{$(Fx)_{j}$}}
	\put (108,15){$\vdots$}

	\put (9,23){\line(0,-1){4}}
    \put (8.5,15){$\vdots$}
 \end{picture}
\end{center} 

\caption{Equivalent analysis schematic for an infinitely iterated dyadic filter bank.} \label{IFBanalysis2}
\end{figure}
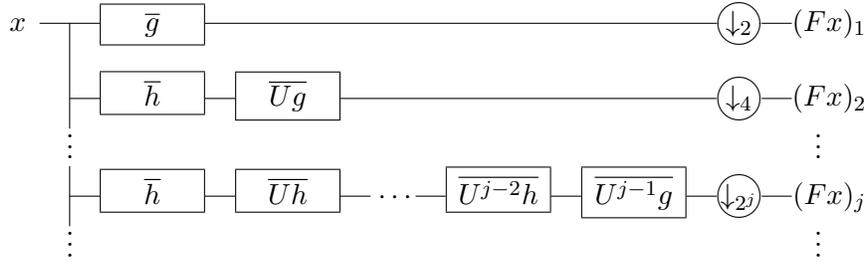

The block diagram of Figure \ref{IFBanalysis2} motivates a notion of iterated low- and high-pass filters that account for the consecutive convolution steps in each channel.  Let $j\in \mathbb{N}$.  Define the \emph{iterated low-pass filter of order $j$} by
\begin{equation} \label{hj-eq}
h_{j} = h * Uh * \cdots * U^{j-1}h,
\end{equation}

\noindent
so that
\begin{equation} \label{hj-def}
\hat{h}_{j}(\xi) = \prod_{l=0}^{j-1} \hat{h}(2^{l}\xi), \quad \xi \in \mathbb{R}.
\end{equation}

\noindent
Similarly, define the \emph{iterated high-pass filter of order $j$} by $g_{j} = h_{j-1}* U^{j-1}g$ so that
\begin{equation} \label{gj-def}
\hat{g}_{j}(\xi) = \hat{g}(2^{j-1}\xi) \, \prod_{l=0}^{j-2} \hat{h}(2^{l}\xi) = \hat{g}(2^{j-1}\xi) \, \hat{h}_{j-1}(\xi), \quad \xi \in \mathbb{R}.
\end{equation}

\noindent
Notice that an iterated low-pass filter $h_{j}$ satisfies $\hat{h}_{j}(0)=2^{\frac{j}{2}}$ and $\hat{h}_{j}(\half) = 0$, while an iterated high-pass filter $g_{j}$ must satisfy $\hat{g}_{j}(0)=0$.  The components $(Fx)_{j}$, $j\in \mathbb{N}$, can now be expressed in terms of the iterated high-pass filters as $(Fx)_{j} = D^{j}(x*\overline{g}_{j})$.  Moreover, notice that
$$ (Fx)_{j}(k) = (x*\overline{g}_{j})(2^{j} k) = \sum_{m\in \mathbb{Z}} x(m) \overline{g_{j}(m-2^{j}k)} = \langle x, T^{2^{j} k} g_{j}\rangle, $$

\noindent
where $T$ denotes the \emph{translation} operator on $\ell^{2}(\mathbb{Z})$ acting on $x\in \ell^{2}(\mathbb{Z})$ by $(Tx)(k)=x(k-1)$.  This calculation shows that the terms of the sequences $(Fx)_{j}$ correspond to inner products of the signal $x$ with specific translates of the iterated high-pass filters.  In light of this observation, the infinitely iterated dyadic filter bank defined by $h$ and $g$ is said to be \emph{stable} when the collection
$$ \lbrace T^{2^{j}k} g_{j} : j\in \mathbb{N}, k\in \mathbb{Z} \rbrace$$

\noindent
constitutes a frame for $\ell^{2}(\mathbb{Z})$.  Equivalently, the infinitely iterated dyadic filter bank generated by $h$ and $g$ is stable if there exist constants $0<A\le B<\infty$ such that for all $x\in \ell^{2}(\mathbb{Z})$,
\begin{equation} \label{IFBframe}
A \Vert x\Vert^{2} \le \sum_{j=1}^{\infty} \Vert (Fx)_{j}\Vert^{2} \le B \Vert x \Vert^{2}.
\end{equation}

\noindent
The latter formulation of stability will be useful for the examination of the Bessel bound, while the former description is better suited to the study of the lower frame bound.

\section{Finitely Iterated Dyadic Filter Banks} \label{FIDFBs}

It was observed in the introduction that Bayram and Selesnick related the frame properties of a finitely iterated dyadic filter bank to those of a related wavelet system.  The purpose of this section is to relate the frame properties of finitely iterated dyadic filter banks to those of the corresponding infinitely iterated dyadic filter bank.  The analysis stage of the \emph{finitely iterated dyadic filter bank of order $j$} generated by a low-pass filter $h$ and a high-pass filter $g$ is depicted in Figure \ref{FIFBanalysis}.  Formally, the \emph{filter bank analysis operator of order $j$} defined by the filters $h, g \in \ell^{2}(\mathbb{Z})$ and acting on $x\in \ell^{2}(\mathbb{Z})$ is the mapping $F_{j}:\ell^{2}(\mathbb{Z}) \rightarrow \bigoplus_{l=1}^{j+1} \ell^{2}(\mathbb{Z})$ defined by
\begin{equation*}
F_{j}: x \mapsto F_{j} x := \lbrace (F_{j} x)_{l} \rbrace_{l=1}^{j+1},
\end{equation*}

\noindent
where $(F_{j} x)_{l} = D^{l}(x*\overline{g}_{l})$ for $1\le l \le j$ and $(F_{j} x)_{j+1} = D^{j}(x*\overline{h}_{j})$.  Notice that $(F_{j} x)_{l} = (Fx)_{l}$ when $l \le j$, while $(F_{j} x)_{j+1}$ accounts for the contribution of the iterated low-pass filter of order $j$.  Following the definition of stability for infinitely iterated dyadic filter banks, the finitely iterated dyadic filter bank is said to be \emph{stable} when there exist constants $0<A\le B <\infty$ such that for all $x\in \ell^{2}(\mathbb{Z})$,
$$ A\Vert x\Vert^{2} \le \sum_{l=1}^{j+1} \Vert (F_{j} x)_{l} \Vert^{2} \le B \Vert x \Vert^{2}.$$

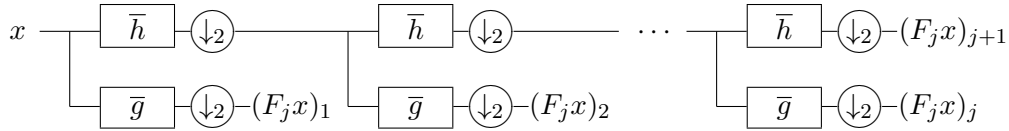
\begin{figure}[hbtp]
\begin{center}
  \begin{picture}(132,15)(1,35)
	\put (1,45){$x$}
	\put (5,46){\line(1,0){8}}

	\put (9,46){\line(0,-1){10}}
    \put (13,45){\framebox[10mm]{$\overline{h}$}}
	\put (23,46){\line(1,0){2}}
	\put (28,46){\circle{6}}
	\put (25,45){\makebox[6mm]{$\downarrow_{2}$}}
	\put (31,46){\line(1,0){19}}
	\put (9,36){\line(1,0){4}}
    \put (13,35){\framebox[10mm]{$ \overline{g}$}}
	\put (23,36){\line(1,0){2}}
	\put (28,36){\circle{6}}
	\put (25,35){\makebox[6mm]{$\downarrow_{2}$}}
	\put (31,36){\line(1,0){2}}
	\put (34.5,35){\makebox[8mm]{$(F_{j} x)_{1}$}}

	\put (46,46){\line(0,-1){10}}
    \put (50,45){\framebox[10mm]{$\overline{h}$}}
	\put (65,46){\circle{6}}
	\put (62,45){\makebox[6mm]{$\downarrow_{2}$}}
	\put (68,46){\line(1,0){14}}
	\put (60,46){\line(1,0){2}}
	\put (46,36){\line(1,0){4}}
    \put (50,35){\framebox[10mm]{$ \overline{g}$}}
	\put (60,36){\line(1,0){2}}
	\put (65,36){\circle{6}}
	\put (62,35){\makebox[6mm]{$\downarrow_{2}$}}
	\put (68,36){\line(1,0){2}}
	\put (71.5,35){\makebox[8mm]{$(F_{j} x)_{2}$}}

    \put (83,45){\makebox[8mm]{$\cdots$}}

	\put (91,46){\line(1,0){8}}
	\put (95,46){\line(0,-1){10}}
    \put (99,45){\framebox[10mm]{$\overline{h}$}}
	\put (114,46){\circle{6}}
	\put (111,45){\makebox[6mm]{$\downarrow_{2}$}}
	\put (109,46){\line(1,0){2}}
	\put (95,36){\line(1,0){4}}
    \put (99,35){\framebox[10mm]{$ \overline{g}$}}
	\put (109,36){\line(1,0){2}}
	\put (114,36){\circle{6}}
	\put (111,35){\makebox[6mm]{$\downarrow_{2}$}}
	\put (117,46){\line(1,0){2}}
	\put (122.5,45){\makebox[8mm]{$(F_{j} x)_{j+1}$}}
	\put (117,36){\line(1,0){2}}
	\put (120.5,35){\makebox[8mm]{$(F_{j} x)_{j}$}}
 \end{picture}
\end{center} 

\caption{Analysis schematic for a finitely iterated dyadic filter bank.} \label{FIFBanalysis}
\end{figure}

The first result describes a situation in which the stability of the infinitely iterated dyadic filter bank guarantees the stability of the finitely iterated dyadic filter bank for any number of iterations and with uniform bounds.

\begin{lemma} \label{A1B}
Let $h,g\in \ell^{2}(\mathbb{Z})$.  Suppose that the infinitely iterated dyadic filter bank generated by $h$ and $g$ is stable, with bounds $0<A\le 1 \le B<\infty$.  Then, any finitely iterated dyadic filter bank generated by $h$ and $g$ is stable with bounds $A/B$, $B/A$.
\end{lemma}

\begin{proof}
Let $x\in \ell^{2}(\mathbb{Z})$ and fix $j\in \mathbb{N}$.  Throughout the proof, the components of $(F_{j} x)_{l}$, $1\le l \le j$, will be denoted by $y_{l}$ and $(F_{j}x)_{j+1}$ will be denoted by $x_{j}$.  Notice that $Fx= \lbrace y_{l}\rbrace_{l=1}^{j} \oplus Fx_{j}$.  Hence, 
$$ \Vert Fx\Vert^{2} = \sum_{l=1}^{\infty} \Vert (F x)_{l}\Vert^{2} = \Vert (F x_{j}) \Vert^{2} + \sum_{l=1}^{j} \Vert y_{l} \Vert^{2}.$$

\noindent
The hypothesis implies that $A \Vert x\Vert^{2} \le \Vert F x\Vert^{2} \le B \Vert x\Vert^{2}$. Since $A\Vert x_{j}\Vert^{2} \le \Vert F x_{j}\Vert^{2} \le B \Vert x_{j}\Vert^{2}$ and $A\le 1$, it follows that
$$ \Vert F_{j}x\Vert^{2} = \Vert x_{j}\Vert^{2} + \sum_{l=1}^{j} \Vert y_{l}\Vert^{2} \le \frac{1}{A} \Vert F x_{j}\Vert^{2} + \sum_{l=1}^{j} \Vert y_{l}\Vert^{2} = \frac{1}{A} \left ( \Vert F x\Vert^{2}-\sum_{l=1}^{j} \Vert y_{l}\Vert^{2} \right ) + \sum_{l=1}^{j} \Vert y_{l}\Vert^{2} \le \frac{B}{A} \Vert x\Vert^{2}.$$ 

\noindent
Similarly, since $B\ge 1$, one has
$$ \Vert F_{j}x\Vert^{2} \ge \frac{1}{B} \Vert F x_{j}\Vert^{2} + \sum_{l=1}^{j} \Vert y_{l}\Vert^{2} = \frac{1}{B} \left ( \Vert Fx\Vert^{2} - \sum_{l=1}^{j} \Vert y_{l}\Vert^{2} \right ) + \sum_{l=1}^{j} \Vert y_{l}\Vert^{2} \ge \frac{A}{B} \Vert x\Vert^{2}.$$

\noindent
Thus the finitely iterated dyadic filter bank of order $j$ is stable with bounds $\frac{A}{B}$ and $\frac{B}{A}$, as claimed.
\end{proof}

The assumption that $A\le 1\le B$ in Lemma \ref{A1B} is critical to the proof, but not the result, leading to the following theorem.

\begin{theorem} \label{IIFBtoFIFB}
Let $h,g\in \ell^{2}(\mathbb{Z})$.  If the infinitely iterated dyadic filter bank generated by $h$ and $g$ is stable with bounds $A$ and $B$, then the finitely iterated dyadic filter bank generated by $h$ and $g$ is stable for any number of iterations with bounds $\min{\lbrace A,A/B\rbrace}$ and $\max{\lbrace B, B/A\rbrace }$.
\end{theorem}

\begin{proof}
If $A\le 1\le B$, then Lemma \ref{A1B} guarantees that the finitely iterated dyadic filter bank is stable with bounds $A/B$ and $B/A$.  It remains to consider the cases where $A\ge 1$ or $B\le 1$.

Assume that $h$ and $g$ generate a stable infinitely iterated dyadic filter bank for which $A\ge 1$.  Define $g'=A^{-\frac{1}{2}} g$ and observe that the infinitely iterated dyadic filter bank generated by $h$ and $g'$ is stable with bounds $A'=1\le B'=B/A$.  Lemma \ref{A1B} implies that any finitely iterated dyadic filter bank generated by $h$ and $g'$ is stable with bounds $A/B$ and $B/A$, i.e.,
$$ \frac{A}{B} \Vert x\Vert^{2} \le \Vert (F_{j}x)_{j+1}\Vert^{2} + \frac{1}{A} \sum_{l=1}^{j} \Vert (F_{j} x)_{l} \Vert^{2} \le \frac{B}{A} \Vert x\Vert^{2}.$$

\noindent
Because $A\ge 1$, the left-hand inequality leads immediately to the lower bound $A/B$ for the finite filter bank associated with $h$ and $g$.  Moreover, multiplying the right-hand inequality by $A$ leads to 
$$ A \Vert (F_{j}x)_{j+1}\Vert^{2} + \sum_{l=1}^{j} \Vert (F_{j} x)_{l} \Vert^{2} \le B \Vert x\Vert^{2},$$

\noindent
which, because $A\ge 1$, leads to the upper bound $B$ for the finite filter bank associated with $h$ and $g$.




An analogous argument in the case $B\le 1$ using $g'=B^{-\frac{1}{2}} g$ leads to the bounds $A$ and $B/A$ for the finitely iterated dyadic filter bank of order $j$ associated with $h$ and $g$.  Combining the three cases, it follows that whenever the infinitely iterated dyadic filter bank is stable with bounds $A$ and $B$, the finitely iterated dyadic filter bank of order $j$ must be stable with bounds $\min{\lbrace A, A/B\rbrace }$ and $\max{\lbrace B,B/A\rbrace}$.
\end{proof}

It is interesting that neither the proof of Lemma \ref{A1B} nor that of Theorem \ref{IIFBtoFIFB} make explicit use of the defining properties of the low-pass or high-pass filter.  This will not be the case for results which derive the stability of the infinitely iterated dyadic filter bank from that of the finitely iterated dyadic filter bank (with uniform bounds), as shown by the following proposition.

\begin{proposition} \label{xjzero-1}
Let $h,g\in \ell^{2}(\mathbb{Z})$ and suppose that the infinitely iterated dyadic filter bank generated by $h$ and $g$ is stable.  Then, for every $x\in \ell^{2}(\mathbb{Z})$, $\Vert (F_{j} x)_{j+1}\Vert \rightarrow 0$ as $j\rightarrow \infty$.
\end{proposition}

\begin{proof}
Assume that the infinitely iterated dyadic filter bank associated with $h$ and $g$ is stable with bounds $A$ and $B$.  Fix $x\in \ell^{2}(\mathbb{Z})$ and let $x_{j}:=(F_{j}x)_{j+1}$ and $y_{j}:=(Fx)_{j}$ for $j\in \mathbb{N}$.  Observe that
$$ \Vert Fx\Vert^{2} = \sum_{l=1}^{j} \Vert y_{l} \Vert^{2} + \Vert F x_{j} \Vert^{2} \ge \sum_{l=1}^{j} \Vert y_{l} \Vert^{2} + A \Vert x_{j} \Vert^{2}.$$

\noindent
It follows that
$$ \Vert x_{j}\Vert^{2} \le \frac{1}{A} \sum_{l=j+1}^{\infty} \Vert y_{l}\Vert^{2},$$

\noindent
which tends to zero as $j\rightarrow \infty$ since
$$ \sum_{l=1}^{\infty} \Vert y_{l}\Vert^{2} \le B \Vert x\Vert^{2}.$$
\end{proof}

Proposition \ref{xjzero-1} describes a necessary condition on the low-pass filter $h$ that an infinitely iterated dyadic filter bank associated with $h$ (and some high-pass filter) is stable.  It turns out that this condition is sufficient for the stability of a finitely iterated dyadic filter bank (with uniform bounds) to imply that of the associated infinitely iterated dyadic filter bank.

\begin{theorem} \label{FIFBtoIIFB}
Let $h,g\in \ell^{2}(\mathbb{Z})$ with $h$ a low-pass filter and $g$ a high-pass filter.

\begin{enumerate}[leftmargin=0.25in,itemsep=0.125in,label = (\alph*)]
\item If the finitely iterated dyadic filter bank generated by $h$ and $g$ is Bessel for any number of iterations with bound $B$, then the infinitely iterated dyadic filter bank associated with $h$ and $g$ is also Bessel with bound $B$.

\item If the finitely iterated dyadic filter bank generated by $h$ and $g$ is stable for any number of iterations with bounds $A, B$ and, for all $x\in \ell^{2}(\mathbb{Z})$, $\Vert (F_{j})_{j+1}x\Vert\rightarrow 0$ as $j\rightarrow \infty$, then the infinitely iterated dyadic filter bank generated by $h$ and $g$ is stable with bounds $A$ and $B$.
\end{enumerate}
\end{theorem}

\begin{proof}
Fix $x\in \ell^{2}(\mathbb{Z})$.  If the finitely iterated dyadic filter bank generated by $h$ and $g$ is Bessel for any number of iterations with bound $B$, then
$$ \sum_{l=1}^{j} \Vert (F x)_{l}\Vert^{2} = \sum_{l=1}^{j} \Vert (F_{j}x)_{l}\Vert^{2} \le B \Vert x\Vert^{2}$$

\noindent
for each $j\in \mathbb{N}$.  It follows that $\Vert Fx\Vert^{2} \le B \Vert x\Vert^{2}$, which completes the proof of (a).

Now assume that finitely iterated dyadic filter bank generated by $h$ and $g$ is stable with bounds $A,B$ for any number of iterations and that $\Vert (F_{j}x)_{j+1}\Vert \rightarrow 0$ as $j\rightarrow \infty$.  The upper bound follows from (a).  The stability of the finitely iterated filter bank implies that
$$ \sum_{l=1}^{j} \Vert (F x)_{l}\Vert^{2} \ge A\Vert x\Vert^{2} - \Vert (F_{j} x)_{j+1}\Vert^{2}$$

\noindent
for each $j\in \mathbb{N}$.  The additional assumption that $\Vert (F_{j} x)_{j+1}\Vert \rightarrow 0$ as $j\rightarrow \infty$ thus guarantees that
$$ \sum_{l=1}^{\infty} \Vert (F x)_{l}\Vert^{2} \ge A\Vert x\Vert^{2},$$

\noindent
completing the proof.
\end{proof}

In light of Proposition \ref{xjzero-1} and Theorem \ref{FIFBtoIIFB} it is natural to consider what properties of a low-pass filter $h$ will guarantee that $\Vert (F_{j}x)_{j+1}\Vert \rightarrow 0$ as $j\rightarrow 0$.  Towards this end, define $H:\ell^{2}(\mathbb{Z}) \rightarrow \ell^{2}(\mathbb{Z})$ by
\begin{equation} \label{Hop}
 Hx = D(x*h),
\end{equation}

\noindent
so that $(F_{j}x)_{j+1} = H^{j}x$.  Let $U_{L}$, $L\in \mathbb{N}$, represent the subspace of $\ell^{2}(\mathbb{Z})$ consisting of sequences supported on $\lbrace k\in \mathbb{Z} : \vert k \vert \le L \rbrace$.  Notice that if $h\in U_{L}$, then $U_{L}$ is an invariant subspace under $H$ and, moreover, for any finitely supported sequence $x$, $H^{j}x$ will belong to $U_{L}$ for sufficiently large $j$.  Notice that if $h,x\in U_{L}$, then
$$ (Hx)(k) = \sum_{m=-L}^{L} h(m) x(2k-m)$$

\noindent
can be nonzero only when $\vert 2k-m\vert \le L$.  Hence, $H$ has the following matrix representation on $U_{L}$,
\setstackgap{L}{1.1\baselineskip}
\setstacktabbedgap{1pt}
\fixTABwidth{T}
\TABstackMath
\begin{equation*}
\resizebox{0.95\hsize}{!}{$
\begin{bmatrix} Hx(-L) \\ Hx(-L+1) \\ \vdots \\ H x(0) \\ \vdots \\ H x(L-1) \\ H x(L) \end{bmatrix} = \bracketMatrixstack{
h(-L) & 0 & 0 & 0 & \cdots & 0 & 0 & 0 & 0 \\
h(-L+2) & h(-L+1) & h(-L) & 0 & \cdots & 0 & 0 & 0 & 0 \\
\vdots   &  &  & & & &  & & \vdots \\
h(L) & h(L-1) & h(L-2) & h(L-3) & \cdots & h(-L+3)  & h(-L+2) & h(-L+1) & h(-L) \\
\vdots   &  & & &  &  &  & & \vdots \\
 0 & 0 & 0 & 0 & \cdots & 0 & h(L) & h(L-1) & h(L-2) \\
 0 & 0 & 0 & 0 & \cdots & 0 & 0 & 0 & h(L) }
\begin{bmatrix} x(-L) \\ x(-L+1) \\ \vdots \\ x(0) \\ \vdots \\ x(L-1) \\ x(L) \end{bmatrix}.$}
\end{equation*}

\noindent
The following proposition describes a simple sufficient condition on a finitely supported low-pass filter that $H$ is a contraction on $U_{L}$.

\begin{proposition} \label{hkpositive}
Let $h\in U_{L}$ be a low-pass filter and suppose that $h(k)\ge 0$ for each $\vert k\vert \le L$.  Then, $H$ is a strict contraction on $U_{L}$.
\end{proposition}

\begin{proof}
On $U_{L}$, the operator $H$ can be represented by a $(2L+1)\times (2L+1)$ matrix, as illustrated earlier.  Observe that the column sums of this matrix must equal either
$$ \sum_{k\in \mathbb{Z}} h(2k) \qquad \text{or} \qquad \sum_{k\in \mathbb{Z}} h(2k+1).$$

\noindent
However, because $h$ is a low-pass filter, it follows that $\hat{h}(0)=\sqrt{2}$ and $\hat{h}(\half)=0$.  Thus
$$ \sum_{k=-L}^{L} h(k) = \sqrt{2} \qquad \text{and} \qquad \sum_{k=-L}^{L} (-1)^{k} h(k) = 0.$$

\noindent
It follows that
$$ \sum_{k\in \mathbb{Z}} h(2k) = \sum_{k\in \mathbb{Z}} h(2k+1) = \frac{1}{\sqrt{2}},$$

\noindent
so that every column sum is equal to $\frac{1}{\sqrt{2}}$.  Moreover, because each entry in the matrix is non-negative, 
$$ \sum_{k} \vert H_{kj}\vert = \frac{1}{\sqrt{2}}$$

\noindent
for each column $j$.  Therefore, the Gershgorin disk theorem implies that any eigenvalue $\lambda$ of $H$ (restricted to $U_{L}$) must satisfy $\vert \lambda \vert \le \frac{1}{\sqrt{2}}$, showing that $H$ is a strict contraction on $U_{L}$.
\end{proof}

Proposition \ref{hkpositive} describes a specific scenario in which the stability of the finitely iterated filter banks (with uniform bounds) ensures the stability of the infinitely iterated filter bank.

\begin{corollary} 
Let $h,g \in \ell^{2}(\mathbb{Z})$ with $h$ a low-pass filter and $g$ a high-pass filter.  Assume that $h\in U_{L}$ and $h(k)\ge 0$ for each $\vert k\vert \le L$.  If the finitely iterated dyadic filter bank of order $j$ generated by $h$ and $g$ is stable with bounds $A$ and $B$ (independent of $j$), then the infinitely iterated dyadic filter bank generated by $h$ and $g$ is stable with bounds $A$ and $B$.
\end{corollary}

\begin{proof}
Fix $\varepsilon >0$.  Let $x\in \ell^{2}(\mathbb{Z})$ and choose $x_{0}\in \ell^{2}(\mathbb{Z})$ to be finitely supported and satisfy $\Vert x-x_{0}\Vert < \varepsilon$.  The fact that $h\in U_{L}$ implies that $H^{j}x_{0} \in U_{L}$ for sufficiently large $j$.  Thus, by Proposition \ref{hkpositive}, it follows that $\Vert (F_{j} x_{0})_{j+1}\Vert \rightarrow 0$ as $j\rightarrow \infty$.  Moreover, because the finitely iterated dyadic filter bank of order $j$ is stable with bounds $A$ and $B$ (for any $j$), it follows that
$$ \lim_{j\rightarrow \infty} \Vert (F_{j} x)_{j+1}\Vert \le \lim_{j\rightarrow \infty} \left ( \Vert (F_{j} x-x_{0})_{j+1}\Vert + \Vert (F_{j} x_{0})_{j+1}\Vert \right ) \le \sqrt{B}\varepsilon.$$ 

\noindent
Because $\varepsilon$ was chosen arbitrarily, it follows that $\Vert (F_{j} x)_{j+1}\Vert \rightarrow 0$ as $j\rightarrow 0$.  Hence, by Theorem \ref{FIFBtoIIFB}, the infinitely iterated filter bank generated by $h$ and $g$ is stable with bounds $A$ and $B$.
\end{proof}

\section{A Sufficient Condition for the Bessel Bound} \label{Bessel}

The goal of this section is to give sufficient conditions for an infinitely iterated dyadic filter bank to satisfy a Bessel bound.  In some cases, this can be derived from properties of the associated scaling function and wavelet, see \cite[Theorem 1]{BayramSelesnick2009}.  However, the goal of this work is to avoid any such assumptions.

The dyadic downsampling and upsampling operations on $\ell^{2}(\mathbb{Z})$ play the role of \emph{dilation} with an iterated filter bank, motivating the examination of the Fourier transforms of the iterated filters on dyadic annuli.  Let  $\mathbb{A}_{l}$, $l\in \mathbb{N}$, denote the dyadic annulus
$$ \mathbb{A}_{l} = \lbrace \xi \in \mathbb{T} : 2^{-(l+1)}< \vert \xi\vert \le 2^{-l}\rbrace.$$

\noindent
It follows that for any $x\in \ell^{2}(\mathbb{Z})$,
$$ \Vert x\Vert^{2} = \sum_{l=1}^{\infty} \Vert \hat{x} \cdot \mathbbm{1}_{\mathbb{A}_{l}} \Vert^{2}.$$

The next two lemmas are adapted from many constructions of both orthonormal and biorthogonal wavelets \cite{Cohen1992, CohenDaubechiesFeauveau1992, Daubechies1992} in which the low-pass filter assumes the form
$$ \hat{h}(\xi) = \sqrt{2} \left ( \frac{1+e^{2\pi i\xi}}{2} \right )^{n} p(\xi),$$

\noindent
where $n\in \mathbb{N}$ and $p$ is a trigonometric polynomial satisfying $p(0)=1$.  In particular, the essence of each lemma can be found in  \cite[Proposition 4.8]{CohenDaubechiesFeauveau1992} and its proof.  The estimates provided by these lemmas will be used to bound $\vert \hat{h}_{j}\vert$ on the annuli $\mathbb{A}_{l}$, which facilitates the derivation of Bessel bounds for the iterated filter bank.  

\begin{lemma} \label{sinelemma}
Let $j$ be a positive integer, then for $\xi \in \lbrack -\frac{1}{2},\frac{1}{2}\rbrack$, 
\begin{equation} \label{sine-product}
\left \vert \prod_{k=0}^{j-1} \frac{1+e^{2\pi i 2^{k}\xi}}{2} \right \vert \le \min{\left \lbrace 1, \frac{1}{2^{j+1} \vert \xi \vert} \right \rbrace}.
\end{equation}
\end{lemma}

\begin{proof}
Assume throughout the proof that $\xi \in \lbrack -\frac{1}{2}, \frac{1}{2} \rbrack$.  The finite product \eqref{sine-product} can be calculated using the trigonometric identity $2\sin{x} \cos{x} = \sin{(2x)}$ and a telescoping argument as follows:
$$ \left \vert \prod_{k=0}^{j-1} \frac{1+e^{2\pi i 2^{k}\xi}}{2} \right \vert = \prod_{k=0}^{j-1} \left \vert \cos{(\pi 2^{k}\xi)} \right \vert = \prod_{k=0}^{j-1} \left \vert \frac{ \sin{(\pi 2^{k+1}\xi)} }{2 \sin{(\pi 2^{k}\xi)} } \right \vert = \left \vert \frac{\sin{(\pi 2^{j} \xi)}}{2^{j} \sin{(\pi \xi)}} \right \vert.$$

\noindent
The numerator of this expression is obviously bounded by one, so the next step is to bound the denominator away from zero.  Observe that for any $\xi \in \lbrack -\frac{1}{2}, \frac{1}{2}\rbrack$ one has $ \vert \sin{(\pi \xi)} \vert \ge 2\vert \xi \vert$ and thus
$$ \left \vert \prod_{k=0}^{j-1} \frac{1+e^{2\pi i 2^{k}\xi}}{2} \right \vert \le \frac{1}{2^{j+1} \vert \xi\vert }.$$

\noindent
However, \eqref{sine-product} is also a product of factors which have modulus at most one, so 
$$ \left \vert \prod_{k=0}^{j-1} \frac{1+e^{2\pi i 2^{k}\xi}}{2} \right \vert \le 1$$

\noindent
for all $\xi \in \lbrack -\frac{1}{2}, \frac{1}{2} \rbrack$.
\end{proof}

It is evident from Lemma \ref{sinelemma} how the ``cosine factor'' in the low-pass filter leads to a natural decay for $\vert \hat{h}_{j} \vert$ on the annuli $\mathbb{A}_{l}$ with $l \le j$.  The next lemma focuses on the second factor, $p(\xi)$, which must be controlled adequately for the iterated filter bank to yield a Bessel bound.

\begin{lemma} \label{p-lemma}
Let $p$ be a trigonometric polynomial satisfying $p(0)=1$.  Assume that there exists $s\in \mathbb{N}$ and $0<\varepsilon <n$ such that
\begin{equation} \label{p-condition}
\sup_{\xi \in \mathbb{\mathbb{R}}} \left \vert \prod_{k=0}^{s-1}  p(2^{k}\xi) \right \vert \le 2^{(n-\varepsilon)s}.
\end{equation}

\noindent
Then, there exists $C_{1}>0$ such that for all $j,l \in \mathbb{N}$,
\begin{equation} \label{Bessel-est}
\sup_{\xi \in \mathbb{A}_{l}} \left \vert \prod_{k=0}^{j-1} p(2^{k} \xi) \right \vert \le \begin{cases} C_{1} 2^{(j-l)(n-\varepsilon)} & 1\le l \le j, \\ C_{1} & l > j. \end{cases} 
\end{equation}
\end{lemma}

\begin{proof}
Write $p(\xi) = \sum_{m\in \mathbb{Z}} p_{m} e^{-2\pi i m \xi}$ and, using an argument of Daubechies \cite[\S 6.2, p. 175]{Daubechies1992}, observe that
$$ \vert p(\xi) \vert \le 1 + \vert p(\xi)-1\vert = 1 + \left \vert \sum_{m\in \mathbb{Z}} p_{m} \left ( e^{-2\pi i m \xi} -1 \right ) \right \vert \le 1 + \sum_{m\in \mathbb{Z}} 2\vert p_{m} \sin{(m\pi \xi)} \vert.$$

\noindent
Only finitely many $p_{m}$ are nonzero, so the estimate $\vert \sin{ x}\vert \le \vert x\vert $ leads to 
$$ \vert p(\xi) \vert \le 1+ C_{2}\vert \xi\vert \le e^{C_{2}\vert \xi \vert}$$    

\noindent
for some $C_{2}>0$.  If, for some $K\in \mathbb{N}$, $2^{K-1}\vert \xi\vert \le \frac{1}{2}$, then
$$ \prod_{k=0}^{K-1} \vert p(2^{k}\xi) \vert \le e^{C_{2} (\vert \xi\vert + 2\vert \xi \vert + \cdots + 2^{K-1}\vert \xi \vert)} \le e^{C_{2}} = C_{3}.$$

\noindent
Therefore, if $\xi \in \mathbb{A}_{l}$ with $l \ge j$, it follows that $2^{j-1}\vert \xi \vert \le \frac{1}{2}$ and thus
$$ \prod_{k=0}^{j-1} \vert p(2^{k}\xi) \vert \le C_{3}.$$

\noindent
However, if $\xi \in \mathbb{A}_{l}$ with $1\le l < j$, then $2^{l-1}\vert \xi \vert \le \frac{1}{2}$, but $2^{k}\vert \xi\vert > \frac{1}{2}$ for $k\ge l$.  In this case, the factors with $k\ge l$ must be controlled using \eqref{p-condition}.  There are $j-l$ such factors and it is possible to organize the factors into at most $\lfloor (j-l)/s \rfloor$ groups of size $s$, excluding at most $s-1$ factors.  The product of any left-over factors can be bounded by some $C_{4}>1$, leading to the estimate
$$ \prod_{k=0}^{j-1} \vert p(2^{k}\xi) \vert \le C_{3} C_{4} 2^{(j-l)(n-\varepsilon)}$$

\noindent
for $\xi \in \mathbb{A}_{l}$, $1\le l < j$.  Letting $C_{1}$ equal $C_{3} C_{4}$ completes the proof of \eqref{Bessel-est}.
\end{proof}

Recall from Section \ref{preliminaries} that the filter bank analysis operator produces sequences $(Fx)_{j}=D^{j}(x*\overline{g}_{j})$, $j\in \mathbb{Z}$.  Given any sequence $x\in \ell^{2}(\mathbb{Z})$, one can decompose its Fourier transform as
$$ \hat{x}(\xi) = \sum_{l \in \mathbb{N}} (\hat{x}\cdot \mathbbm{1}_{\mathbb{A}_{l}}) (\xi), \quad \xi \in \mathbb{T}.$$

\noindent
The next proposition examines the contributions stemming from each dyadic annulus in this decomposition under downsampling.

\begin{proposition} \label{downestimate}
Let $x_{\mathbb{A}_{l}}$, $l \in \mathbb{N}$, denote the sequence in $\ell^{2}(\mathbb{Z})$ such that $\hat{x}_{\mathbb{A}_{l}}(\xi) = \vert \hat{x} \cdot \mathbbm{1}_{\mathbb{A}_{l}}(\xi)\vert $.  Then, for each $j\in \mathbb{N}$,
\begin{equation} \label{annuli-estimate}
\Vert D^{j} x_{\mathbb{A}_{l}} \Vert^{2} \le \begin{cases} 2^{-l} \Vert x_{\mathbb{A}_{l}} \Vert^{2} & 1\le l \le j, \\ 2^{-j} \Vert x_{\mathbb{A}_{l}} \Vert^{2} & l > j. \end{cases}
\end{equation}

\noindent
with equality when $l \ge j$.
\end{proposition}

\begin{proof}
Observe that
$$ \Vert D^{j}x_{\mathbb{A}_{l}}\Vert^{2} = \int_{\mathbb{T}} \left \vert 2^{-j} \sum_{k=0}^{2^{j}-1} \hat{x}_{\mathbb{A}_{l}}(2^{-j}(\xi+k)) \right \vert^{2} \; d\xi.$$

\noindent
If $l \ge j$, then at most one of the terms in the sum is nonzero for each $\xi$, leading to
\begin{align*}
\Vert D^{j} x_{\mathbb{A}_{l}} \Vert^{2} &= \int_{0}^{1} 2^{-2j} \sum_{k=0}^{2^{j}-1} \vert  \hat{x}_{\mathbb{A}_{l}}  (2^{-j}(\xi+k)) \vert^{2} \; d\xi \\
\text{($\xi\mapsto 2^{j}\xi$)} \quad &= 2^{-j} \sum_{k=0}^{2^{j}-1} \int_{0}^{2^{-j}}  \vert \hat{x}_{\mathbb{A}_{l}} (\xi+2^{-j} k) \vert^{2} \; d\xi \\
&= 2^{-j} \Vert x_{\mathbb{A}_{l}} \Vert^{2}.
\end{align*} 

\noindent
If $1\le l <j$, then $\Vert D^{j}x_{\mathbb{A}_{l}}\Vert^{2} = \Vert D^{j-l} D^{l} x_{\mathbb{A}_{l}}\Vert^{2}$ and, because downsampling is norm-reducing, it follows that
$$ \Vert D^{j} x_{\mathbb{A}_{l}} \Vert^{2} \le \Vert D^{l} x_{\mathbb{A}_{l}} \Vert^{2} = 2^{-l} \Vert x_{\mathbb{A}_{l}} \Vert^{2}, \quad 1\le l < j.$$
\end{proof}

Combining Proposition \ref{downestimate} with Lemmas \ref{sinelemma} and \ref{p-lemma} leads to a sufficient condition for the chosen class of filter banks to possess a Bessel bound.  The next theorem is an analog of a result due to Cohen, Daubechies, and Feauveau \cite[Lemma 3.4 and Proposition 4.8]{CohenDaubechiesFeauveau1992} in which the same assumptions on the low- and high-pass filters are shown to guarantee that the corresponding wavelet system is a Bessel sequence in $L^{2}(\mathbb{R})$.

\begin{theorem}[Bessel Bounds] \label{BesselBound}
Let $h\in \ell^{2}(\mathbb{Z})$ be a finitely supported sequence of the form
\begin{equation} \label{haar-type}
\hat{h}(\xi) = \sqrt{2} \left \lbrack \frac{1+e^{2\pi i\xi}}{2} \right \rbrack^{n} p(\xi),
\end{equation}

\noindent
where $n\in \mathbb{N}$ and $p$ is a trigonometric polynomial satisfying $p(0)=1$.  Let $g\in \ell^{2}(\mathbb{Z})$ be finitely supported and satisfy $\hat{g}(0)=0$.  If there exists $s\in \mathbb{N}$ such that
\begin{equation} \label{p-est}
\sup_{\xi \in \mathbb{\mathbb{R}}} \left \vert \prod_{k=0}^{s-1}  p(2^{k}\xi) \right \vert < 2^{(n-\frac{1}{2})s},
\end{equation}

\noindent
then the infinitely iterated filter bank generated by $h$ and $g$ is Bessel.
\end{theorem}

\begin{proof}
The idea of the proof is to split $\vert \widehat{(Fx)}_{j}(\xi) \vert^{2}$ across the dyadic annuli $\mathbb{A}_{l}$, $l \in \mathbb{N}$, and relate the pieces to terms in the sum
$$ \Vert x\Vert^{2} = \sum_{l=1}^{\infty} \Vert x_{\mathbb{A}_{l}} \Vert^{2}.$$

\noindent
The argument will be divided into a number of steps in order to improve the overall clarity of the proof.

\begin{enumerate}[leftmargin=0.25in,itemsep=0.125in,label = \arabic*.]
\item The squared modulus of $\widehat{(Fx)}_{j}$ can be expressed as follows:
\begin{align*}
\vert \widehat{(Fx)}_{j}(\xi) \vert^{2} &= \left \vert 2^{-j} \sum_{k=0}^{2^{j}-1} \hat{x}(2^{-j}(\xi + k)) \overline{\hat{g}_{j} (2^{-j}(\xi + k))} \right \vert^{2} \\
&= \left \vert \sum_{l \in \mathbb{N}} 2^{-j} \sum_{k=0}^{2^{j}-1} (\hat{x}\cdot \overline{\hat{g}_{j}} \cdot \mathbbm{1}_{\mathbb{A}_{l}} )(2^{-j}(\xi + k)) \right \vert^{2} \le \left \lbrack \sum_{l \in \mathbb{N}} 2^{-j} \sup_{\xi \in \mathbb{A}_{l}} \vert \hat{g}_{j}(\xi) \vert \sum_{k=0}^{2^{j}-1} \vert \hat{x}_{\mathbb{A}_{l}} (2^{-j}(\xi + k)) \vert \right \rbrack^{2}.
\end{align*}

\item A bound on $\vert \hat{g}_{j}\vert$ over $\mathbb{A}_{l}$ will be obtained based on the assumed form of the low- and high-pass filters.  Notice that \eqref{p-est} implies that there exists $\varepsilon > \frac{1}{2}$ such that \eqref{p-condition} holds.  Therefore, in light of \eqref{hj-def}, it follows from \eqref{haar-type}, Lemma \ref{sinelemma}, and Lemma \ref{p-lemma} that
$$ \sup_{\xi \in \mathbb{A}_{l}} \vert \hat{h}_{j-1}(\xi) \vert \le \begin{cases} C_{1} 2^{(j-1)/2} 2^{(l+1-j)\varepsilon} & 1\le l \le j-1, \\ C_{1} 2^{(j-1)/2} & l > j-1.\end{cases} $$

\noindent
Define
$$ G_{j,l} := \sup_{\xi \in \mathbb{A}_{l}} \vert \hat{g}_{j}(\xi) \vert$$

\noindent
and observe that 
$$ G_{j,l} \le \left ( \sup_{\xi \in \mathbb{A}_{l}} \vert \hat{h}_{j-1}(\xi) \vert \right ) \cdot \left ( \sup_{\xi \in \mathbb{A}_{l}} \vert \hat{g}(2^{j-1}\xi)\vert \right ).$$

\noindent
Given that $g$ is finitely supported with $\hat{g}(0)=0$, it follows that there exists $C_{5}>0$ such that $\vert \hat{g}(\xi)\vert \le \min{\lbrace C_{5}, C_{5}\vert\xi\vert \rbrace}$ and, therefore,
$$ \vert \hat{g}(2^{j-1}\xi)\vert \le \min{\lbrace C_{5}, C_{5} 2^{j-1}\vert \xi\vert \rbrace}.$$

\noindent
Combining all of the estimates leads to
$$ G_{j,l} \le \begin{cases} C_{6} 2^{(j-1)/2} 2^{(l+1-j)\varepsilon} & 1\le l \le j-1, \\ C_{6} 2^{(j-1)/2} 2^{j-l-1} & l> j-1, \end{cases}$$

\noindent
where $C_{6}=C_{1}C_{5}$.

\item Choose $\delta$ such that $0<\delta<2\varepsilon-1$ and define $C_{j,l}$ by
$$ C_{j,l} = \begin{cases} 2^{(l+1-j)\delta} & 1\le l \le j-1, \\ 2^{j-l-1} & l > j-1. \end{cases}$$

\noindent
It will be important that $\sum_{l \in \mathbb{N}} C_{j,l} \le C$  (independent of $j$), which is justified by the following calculation:
$$ \sum_{l\in \mathbb{N}} C_{j,l} = \sum_{l=1}^{j-1} 2^{(l+1-j)\delta} + \sum_{l=j}^{\infty} 2^{j-l-1} = \sum_{l=0}^{j-2} 2^{-l \delta} + \sum_{l \in \mathbb{N}} 2^{-l} <\infty.$$

\item The two previous steps can now be incorporated in the initial calculation to show that
\begin{align*}
\vert \widehat{(Fx)}_{j}(\xi) \vert^{2} &\le \left \lbrack \sum_{l \in \mathbb{N}} 2^{-j} \sup_{\xi \in \mathbb{A}_{l}} \vert \hat{g}_{j}(\xi) \vert \sum_{k=0}^{2^{j}-1} \vert \hat{x}_{\mathbb{A}_{l}} (2^{-j}(\xi + k)) \vert \right \rbrack^{2} \\
&=\left \lbrack \sum_{l \in \mathbb{N}} \sqrt{C_{j,l}} \frac{G_{j,l}}{\sqrt{C_{j,l}}} 2^{-j} \sum_{k=0}^{2^{j}-1} \vert \hat{x}_{\mathbb{A}_{l}} (2^{-j}(\xi + k)) \vert \right \rbrack^{2} \\
&\le \left ( \sum_{l\in \mathbb{N}} C_{j,l} \right ) \, \left ( \sum_{l \in \mathbb{N}} \frac{G_{j,l}^{2}}{ C_{j,l}} \left \lbrack 2^{-j} \sum_{k=0}^{2^{j}-1} \vert \hat{x}_{\mathbb{A}_{l}} (2^{-j}(\xi+k))\vert \right \rbrack^{2} \right ) \\
&\le C \left ( \sum_{l \in \mathbb{N}} \frac{G_{j,l}^{2}}{ C_{j,l}} \left \lbrack 2^{-j} \sum_{k=0}^{2^{j}-1} \vert \hat{x}_{\mathbb{A}_{l}} (2^{-j}(\xi+k))\vert \right \rbrack^{2} \right ),
\end{align*}

\noindent
where the Cauchy-Schwarz inequality was used in the next to last step.

\item Integrating $\vert \widehat{(Fx)}_{j}(\xi)\vert^{2}$ over $\mathbb{T}$ leads to an expression for $\Vert (Fx)_{j}\Vert^{2}$ and thus
$$ \sum_{j\in \mathbb{N}} \Vert (Fx)_{j}\Vert^{2} \le C \sum_{l\in \mathbb{N}} \sum_{j=1}^{\infty}  \frac{G_{j,l}^{2}}{C_{j,l}} \Vert D^{j}x_{\mathbb{A}_{l}}\Vert^{2}.$$

\noindent
It remains to incorporate the definition of $C_{j,l}$ as well as the estimates for $g_{j,l}$ and $\Vert D^{j} x_{\mathbb{A}_{l}} \Vert^{2}$.  This leads to
\begin{align*}
\sum_{j\in \mathbb{N}} \Vert (Fx)_{j}\Vert^{2} &\le C' \sum_{l\in \mathbb{N}} \left \lbrack \sum_{j=1}^{l} 2^{j-l-1} \Vert x_{\mathbb{A}_{l}} \Vert^{2} + \sum_{j=l+1}^{\infty} 2^{(l-j)(2\varepsilon-\delta)} 2^{2\varepsilon-\delta} 2^{j-l} \Vert x_{\mathbb{A}_{l}} \Vert^{2} \right \rbrack \\
&= C' \sum_{l \in \mathbb{N}} \left \lbrack \sum_{j=1}^{l} 2^{-j} + 2^{2\varepsilon-\delta} \sum_{j=1}^{\infty} 2^{(1+\delta-2\varepsilon)j} \right \rbrack \Vert x_{\mathbb{A}_{l}} \Vert^{2} \\
\text{(because $2\varepsilon-\delta>1$)} \quad &\le C'' \sum_{l=1}^{\infty} \Vert x_{\mathbb{A}_{l}} \Vert^{2} \\
&= C'' \Vert x\Vert^{2}. \qedhere
\end{align*}
\end{enumerate}
\end{proof}

The hypotheses used in Theorem \ref{BesselBound} also guarantee that the norm of the component of a finitely iterated dyadic filter bank due to the iterated low-pass filter $h_{j}$ tends to zero as $j\rightarrow \infty$.

\begin{proposition} \label{xjshrinks}
Let $h\in \ell^{2}(\mathbb{Z})$ be a finitely supported sequence of the form \eqref{haar-type}, where $n\in \mathbb{N}$ and $p$ is a trigonometric polynomial satisfying $p(0)=1$.  If there exists $s\in \mathbb{N}$ such that \eqref{p-est} holds, then, for all $x\in \ell^{2}(\mathbb{Z})$, 
\begin{enumerate}[leftmargin=0.3275in,itemsep=0.125in,label = (\alph*)]
\item $\Vert (F_{j} x)_{j+1})\Vert \rightarrow 0$ as $j\rightarrow \infty$ and
\item there exists $C'>0$ such that $\Vert (F_{j} x)_{j+1}\Vert^{2} \le C' \Vert x\Vert^{2}$ for each $j\in \mathbb{N}$.
\end{enumerate}
\end{proposition}

\begin{proof}
The structure of the proof follows that of Theorem \ref{BesselBound} using a finite decomposition of $\mathbb{T}$.  Fix $j\in \mathbb{N}$ and write
$$ \mathbb{T} = \bigcup_{l=1}^{j} \mathbb{B}_{l}, \qquad \text{where } \mathbb{B}_{l}=\begin{cases} \mathbb{A}_{l} & 1\le l \le j-1, \\ \bigcup_{m\ge j} \mathbb{A}_{m} & l = j. \end{cases}$$

\noindent
Observe that
\begin{align*}
\Vert (F_{j})_{j+1}\Vert^{2} &= \int_{\mathbb{T}} \left \vert 2^{-j} \sum_{k=0}^{2^{j}-1} \hat{x}(2^{-j}(\xi+k)) \overline{\hat{h}_{j}(2^{-j}(\xi+k))} \right \vert^{2} \; d\xi \\
&\le \int_{\mathbb{T}} \left \lbrack \sum_{l=1}^{j} 2^{-j} \sum_{k=0}^{2^{j}-1} \left \vert ( \hat{x}\cdot \hat{h}_{j}\cdot \mathbbm{1}_{\mathbb{B}_{l}})(2^{-j}(\xi+k)) \right \vert \right \rbrack^{2} \; d\xi \\
&\le \int_{\mathbb{T}} \left \lbrack \sum_{l=1}^{j} \sup_{\xi \in \mathbb{B}_{l}} \vert \hat{h}_{j}(\xi)\vert  2^{-j} \sum_{k=0}^{2^{j}-1} \hat{x}_{\mathbb{B}_{l}} (2^{-j}(\xi+k)) \right \rbrack^{2} \; d\xi, \\
\end{align*}

\noindent
where $\hat{x}_{\mathbb{B}_{l}} = \vert \hat{x}\cdot \mathbbm{1}_{\mathbb{B}_{l}} \vert$.  Based on the assumed form of the low-pass filter and the fact that $p$ satisfies \eqref{p-condition} for some $\varepsilon >\frac{1}{2}$ it is possible to choose $\delta$ such that $0<\delta <2\varepsilon -1$.  By Lemma \ref{sinelemma} it follows that
$$ \left \vert \prod_{k=0}^{j-1} \frac{1+e^{2\pi i 2^{k} \xi}}{2} \right \vert^{n} \le  2^{(l-j)n}, \quad \xi \in \mathbb{B}_{l}, \quad 1\le l \le j.$$

\noindent
Similarly, Lemma \ref{p-lemma} implies that
$$ \left \vert \prod_{k=0}^{j-1}  p(2^{k}\xi) \right \vert \le C_{1} 2^{(j-l)(n-\varepsilon)}, \quad \xi \in \mathbb{B}_{l}, \quad 1\le l \le j.$$

\noindent
Combining these estimates with \eqref{hj-def} and \eqref{haar-type}, one has
\begin{align*}
\Vert (F_{j} x)_{j+1}\Vert^{2} &\le (C_{1})^{2} \int_{\mathbb{T}} \left \lbrack \sum_{l=1}^{j} 2^{(l-j)\varepsilon} 2^{j/2} 2^{-j} \sum_{k=0}^{2^{j}-1} \vert \hat{x}_{\mathbb{B}_{l}} \vert (2^{-j}(\xi+k))\vert \right \rbrack^{2} \; d\xi \\
&\le (C_{1})^{2} \int_{\mathbb{T}} \left \lbrack \sum_{l=1}^{j} 2^{(l-j)\delta/2} 2^{(l-j)(\varepsilon-\delta/2)} 2^{j/2} 2^{-j} \sum_{k=0}^{2^{j}-1} \vert \hat{x}_{\mathbb{B}_{l}} \vert (2^{-j}(\xi+k))\vert \right \rbrack^{2} \; d\xi \\
\text{(Cauchy-Schwarz)} \quad &\le (C_{1})^{2} \int_{\mathbb{T}} \left ( \sum_{l=1}^{j}  2^{(l-j)\delta} \right ) \left ( \sum_{l=1}^{j} 2^{(l-j)(2\varepsilon -\delta)} 2^{j} \left \lbrack 2^{-j} \sum_{k=0}^{2^{j}-1} \vert \hat{x}_{\mathbb{B}_{l}}(2^{-j}(\xi+k)) \right \rbrack^{2} \right ) \; d\xi.
\end{align*}

\noindent
Noting that the first sum is uniformly bounded in $j$ and integrating over $\mathbb{T}$ leads to
\begin{align*}
\Vert (F_{j} x)_{j+1}\Vert^{2} &\le C' \sum_{l=1}^{j} 2^{(l-j)(2\varepsilon -\delta)} \Vert 2^{j} D^{j}x_{\mathbb{B}_{l}}\Vert^{2} \\
\text{(Proposition \ref{downestimate})} \quad &\le C' \sum_{l=1}^{j} 2^{(l-j)(2\varepsilon -\delta-1)} \Vert x_{\mathbb{B}_{l}}\Vert^{2} \\
&= C' \sum_{l=1}^{j-1} 2^{(l-j)(2\varepsilon -\delta-1)} \Vert x_{\mathbb{A}_{l}}\Vert^{2} + C' \sum_{l=j}^{\infty} \Vert x_{\mathbb{A}_{l}} \Vert^{2}.
\end{align*}

\noindent
Fix $\sigma >0$ and $x\in \ell^{2}(\mathbb{Z})$.  Since
$$ \Vert x\Vert^{2} = \sum_{l\in \mathbb{N}} \Vert x_{\mathbb{A}_{l}}\Vert^{2},$$

\noindent
it follows that there exists $L\in \mathbb{N}$ such that
$$ \sum_{l\ge L} \Vert x_{\mathbb{A}_{l}}\Vert^{2}< \sigma.$$

\noindent
The previous calculations imply that
$$ \Vert (F_{j} x)_{j+1}\Vert^{2} \le C' 2^{L-j} \sum_{l=1}^{L-1} 2^{(l-L)(2\varepsilon -\delta-1)} \Vert x_{\mathbb{A}_{l}}\Vert^{2} + C' \sum_{l=L}^{\infty} \Vert x_{\mathbb{A}_{l}} \Vert^{2},$$

\noindent
so that
$$ \lim_{j\rightarrow \infty} \Vert (F_{j} x)_{j+1}\Vert^{2} \le \sigma.$$

\noindent
Since $\sigma$ was chosen arbitrarily, it follows that $\Vert (F_{j} x)_{j+1}\Vert\rightarrow 0$ as $j\rightarrow \infty$, completing the proof of (a).  Notice that (b) is an immediate consequence of (a) by the Principle of Uniform Boundedness.
\end{proof}

\section{A Sufficient Condition for the Lower Frame Bound}

This section examines finitely iterated dyadic filter banks from the vantage of the theory of shift-invariant spaces, as studied by Helson \cite{Helson1964}.  The study of shift-invariant spaces can also be based on the characterization of the commutant of the shift operator that was given by von Neumann \cite{Dixmier1981}.  The extensive study of shift-invariant subspaces of $L^{2}(\mathbb{R})$ produced a general framework for the theory of shift-invariant spaces \cite{BDR1994, BDR1995, Bownik2000, Papadakis2000a, RonShen1995, RonShen1997}, which was later broadened to the theory of shift-invariant and translation-invariant spaces on locally compact abelian groups \cite{BownikRoss2015,CabrelliPaternostro2010,Papadakis2000}.  Recall that the filter bank analysis operator of order $j$ defined in Section \ref{FIDFBs} maps a signal $x\in \ell^{2}(\mathbb{Z})$ to the sequences $(F_{j})_{l}$, $1\le l \le j+1$, where
$$ (F_{j})_{l}(k) = (D^{l}(x*\bar{g}_{l})(k) = \langle x, T^{2^{l}k} g_{l}\rangle, \quad 1\le l \le j,$$

\noindent
and
$$ (F_{j})_{j+1}(k) = (D^{j}(x*\bar{h}_{j})(k) = \langle x, T^{2^{l}k} h_{j}\rangle$$

\noindent
for $k\in \mathbb{Z}$.  It is apparent that the stability of a finitely iterated dyadic filter bank of order $j$ can be understood in terms of the frame properties of a $2^{j}\mathbb{Z}$-shift-invariant system in $\ell^{2}(\mathbb{Z})$.  More specifically, if $\Phi_{j}$ is defined by
\begin{equation} \label{generators}
\Phi_{j} = \lbrace g_{1}, T^{2}g_{1}, \ldots, T^{2^{j}-2}g_{1}, \ldots, g_{l}, T^{2^{l}} g_{l}, \ldots, T^{2^{j}-2^{l}} g_{l}, \ldots, g_{j-1}, T^{2^{j-1}}g_{j-1}, g_{j}, h_{j} \rbrace,
\end{equation}

\noindent
then the finitely iterated dyadic filter bank of order $j$ is stable if and only if the collection $\lbrace T^{2^{j}k} \phi : k\in \mathbb{Z}, \; \phi \in \Phi_{j}\rbrace$ is a frame for $\ell^{2}(\mathbb{Z})$.

At the heart of the shift-invariant theory of $\ell^{2}(\mathbb{Z})$ is the fiberization mapping $\mathcal{T}: \ell^{2}(\mathbb{Z}) \longrightarrow L^{2}(\mathbb{T}, \mathbb{C}^{2^{j}})$ defined by
$$ \mathcal{T}x(\xi) = \left ( 2^{-\frac{j}{2}} \hat{x}( 2^{-j}(\xi + m)) \right )_{0\le m\le 2^{j}-1}, \quad \xi \in \mathbb{T}.$$

\noindent
The mapping $\mathcal{T}$ is an isometric isomorphism \cite[Proposition 3.3]{CabrelliPaternostro2010} and facilitates the study of frame properties for shift-invariant systems in terms of the finite-dimensional frame properties of the corresponding fibers in $\mathbb{C}^{2^{j}}$.  Given a finite collection $\Phi=\lbrace \phi_{m}\rbrace_{n=1}^{N} \subset \ell^{2}(\mathbb{Z})$, the \emph{shift-invariant space with period $2^{j}$ generated by $\Phi$} is defined by
$$ \mathcal{S}_{2^{j}\mathbb{Z}} (\Phi) = \overline{\text{span}} \lbrace T^{2^{j}k} \phi_{n} : k\in \mathbb{Z}, \; 1\le n\le N \rbrace.$$

\noindent
It will be convenient to adopt the following notation to denote the $2^{j}\mathbb{Z}$-shift-invariant system generated by $\Phi$,
$$ E_{2^{j}\mathbb{Z}}(\Phi) = \lbrace T^{2^{j}k}\phi : k\in \mathbb{Z}, \; \phi \in \Phi \rbrace.$$

\noindent
The \emph{pre-Gramian} associated with $\Phi$ is the $2^{j}\times N$ matrix-valued function on $\mathbb{T}$ whose $n$th column is given by $\mathcal{T}\phi_{n}(\xi)$.  The pre-Gramian will be written as $\mathcal{T}_{\Phi}$ so that
$$ (\mathcal{T}_{\Phi})_{mn}(\xi) = 2^{-\frac{j}{2}} \hat{\phi}_{n}(2^{-j}(\xi + m)).$$

\noindent
The \emph{Gramian} associated with $\Phi$ is defined by $\mathcal{G}_{\Phi}(\xi) = \mathcal{T}_{\Phi}^{*}(\xi) \mathcal{T}_{\Phi}(\xi)$.  The Gramian matrix plays the role of the Gramian operator on the fibers and facilitates a characterization of the frame bounds.  The following result follows from a more general theorem due to Cabrelli and Paternostro in the context of locally compact abelian groups \cite[Proposition 4.9]{CabrelliPaternostro2010}.  

\begin{proposition} \label{SIbounds}
Let $\Phi= \lbrace \phi_{n}\rbrace_{n=1}^{N}$ be a finite subset of $\ell^{2}(\mathbb{Z})$.  Fix $j\in \mathbb{N}$ and positive constants $A,B$ such that $0<A\le B$.  Then, 
\begin{enumerate}[leftmargin=0.3275in,itemsep=0.125in,label = (\alph*)]
\item The collection $E_{2^{j}\mathbb{Z}}(\Phi)$ is a Bessel sequence with constant $B$ if and only if 
$$ \langle \mathcal{G}_{\Phi} (\xi) c, c \rangle \le B \Vert c\Vert^{2}$$

\noindent
for almost every $\xi \in \mathbb{T}$ and every $c\in \mathbb{C}^{N}$.

\item The collection $E_{2^{j}\mathbb{Z}}(\Phi)$ is a frame for $\mathcal{S}_{2^{j}\mathbb{Z}}(\Phi)$ with constants $A$ and $B$ if and only if 
$$ A \Vert c\Vert^{2} \le \langle \mathcal{G}_{\Phi} (\xi) c, c \rangle \le B \Vert c\Vert^{2}$$

\noindent
for almost every $\xi \in \mathbb{T}$ and every $c\in \mathbb{C}^{N}$.
\end{enumerate}
\end{proposition}

The frame bounds of multi-channel oversampled filter banks were previously characterized by Cvetkovi\'c and Vetterli in the mid-1990s using the polyphase approach \cite{CvetkovicVetterli1994}.  It is not surprising, but worth mentioning, that the polyphase matrix can be written as the matrix product of a specific paraunitary matrix with the pre-Gramian matrix defined above.  The next lemma establishes a simple condition on the low- and high-pass filters guaranteeing that the shift-invariant space $\mathcal{S}_{2^{j}\mathbb{Z}}(\Phi)$ associated with the finitely iterated dyadic filter bank of order $j$ will coincide with $\ell^{2}(\mathbb{Z})$.  This lemma holds for $h,g \in \ell^{2}(\mathbb{Z})$, but will be stated only for finitely supported filters because the general result is not needed here and its proof relies on the range function.

\begin{lemma} \label{fullspan}
Let $h, g \in \ell^{2}(\mathbb{Z})$ be finitely supported with $h$ a low-pass filter and $g$ a high-pass filter.  If
\begin{equation} \label{notzero}
\hat{h}(\xi/2) \hat{g}(\xi/2+1/2) -\hat{g}(\xi/2) \hat{h}(\xi/2+1/2) \neq 0, \quad \xi \in \mathbb{T},
\end{equation}

\noindent
then $\mathcal{S}_{2^{j}\mathbb{Z}}(\Phi_{j}) = \ell^{2}(\mathbb{Z})$, where $\Phi_{j}$ is defined according to \eqref{generators}.
\end{lemma}

\begin{proof}
It will first be shown that $\mathcal{S}_{2\mathbb{Z}}(\Phi_{1}) = \ell^{2}(\mathbb{Z})$.  Let $\delta_{0}$ represent the element of $\ell^{2}(\mathbb{Z})$ such that $\delta_{0}(0)=1$ and $\delta_{0}(k)=0$ for $k\neq 0$.  It is sufficient to prove that $\delta_{0}$ and $T\delta_{0}$ belong to $\mathcal{S}_{2\mathbb{Z}}(\Phi_{1})$.  Observe that $\delta_{0} \in \mathcal{S}_{2\mathbb{Z}}(\Phi_{1})$ if and only if there exist $c,d \in \ell^{2}(\mathbb{Z})$ such that
$$ \delta_{0} = \sum_{m\in \mathbb{Z}} c(m) T^{2m}h + \sum_{m\in \mathbb{Z}} d(m) T^{2m}g$$

\noindent
or, equivalently, 
$$ 1 = \hat{c}(2\xi) \hat{h}(\xi) + \hat{d}(2\xi) \hat{g}(\xi).$$

\noindent
The $\half$-periodicity of $\hat{c}(2\xi)$ and $\hat{d}(2\xi)$ leads to the matrix equation
$$ \begin{bmatrix} \hat{h}(\xi) & \hat{g}(\xi) \\ \hat{h}(\xi+\half) & \hat{g}(\xi + \half) \end{bmatrix} \begin{bmatrix} \hat{c}(2\xi) \\ \hat{d}(2\xi) \end{bmatrix} = \begin{bmatrix} 1 \\ 1 \end{bmatrix}, \quad \xi \in \mathbb{T}.$$

\noindent
Observe that \eqref{notzero} guarantees a solution for which $c,d \in \ell^{2}(\mathbb{Z})$, since the determinant of the $2\times 2$ matrix must be bounded away from zero.  Similar reasoning shows that $T\delta_{0}$ belongs to $\mathcal{S}_{2\mathbb{Z}}(\Phi_{1})$.

Suppose now that $\mathcal{S}_{2^{j}\mathbb{Z}}(\Phi_{j}) = \ell^{2}(\mathbb{Z})$.  In order to show that $\mathcal{S}_{2^{j+1}\mathbb{Z}}(\Phi_{j+1}) = \ell^{2}(\mathbb{Z})$, it is sufficient to prove that $h_{j}$ and $T^{2^{j}}h_{j}$ belong to $\mathcal{S}_{2^{j+1}\mathbb{Z}}(\Phi_{j+1})$.  In particular, the equation
$$ h_{j}(k) = \sum_{m\in \mathbb{Z}} c(m) T^{2^{j+1}} h_{j+1}(m) + \sum_{m\in \mathbb{Z}} d(m) T^{2^{j+1}} g_{j+1}(m)$$

\noindent
is equivalent to
$$ \hat{h}_{j}(\xi) = c(2^{j+1}\xi) \hat{h}_{j+1}(\xi) + d(2^{j+1}\xi) \hat{g}_{j+1}(\xi), \quad \xi \in \mathbb{T}.$$

\noindent
However, the product formulas for the iterated filters allow for the reduction to
$$ 1 = c(2^{j+1}\xi) \hat{h}(2^{j}\xi) + d(2^{j+1}\xi) \hat{g}(2^{j}\xi), \quad \text{a.e.} \; \xi \in \mathbb{T},$$

\noindent
since $\hat{h}_{j-1}(\xi) \neq 0$ almost everywhere.  Evaluating at $\xi+2^{-j}$ leads to the matrix equation
$$ \begin{bmatrix} \hat{h}(2^{j}\xi) & \hat{g}(2^{j}\xi) \\ \hat{h}(2^{j}(\xi + 1/2)) & \hat{g}(2^{j}(\xi + 1/2)) \end{bmatrix} \begin{bmatrix} \hat{c}(2^{j+1} \xi) \\ \hat{d}(2^{j+1} \xi) \end{bmatrix} = \begin{bmatrix} 1 \\ 1 \end{bmatrix}, \quad \text{a.e.} \; \xi \in \mathbb{T},$$

\noindent
which, after the substitution $2^{j}\xi \mapsto \xi$ is equivalent to 
$$ \begin{bmatrix} \hat{h}(\xi) & \hat{g}(\xi) \\ \hat{h}(\xi + 1/2) & \hat{g}(\xi + 1/2) \end{bmatrix} \begin{bmatrix} \hat{c}(2\xi) \\ \hat{d}(2\xi) \end{bmatrix} = \begin{bmatrix} 1 \\ 1 \end{bmatrix}, \quad \text{a.e.} \; \xi \in \mathbb{T}.$$

\noindent
As above, \eqref{notzero} guarantees that the sequences $c$ and $d$ will belong to $\ell^{2}(\mathbb{Z})$, showing that $h_{j} \in \mathcal{S}_{2^{j+1}\mathbb{Z}}(\Phi_{j+1})$.  Similar reasoning applies to $T^{2^{j}}h_{j}$.  Combining this inductive step with the fact that $\mathcal{S}_{2\mathbb{Z}}(\Phi_{1}) = \ell^{2}(\mathbb{Z})$ completes the proof.
\end{proof}

The condition \eqref{notzero} of Lemma \ref{fullspan} does not represent a significant restriction on the class of suitable filters for the construction of stable infinitely iterated filter banks.  In fact, if the high-pass filter $g$ is chosen according to \eqref{stdHP}, then \eqref{notzero} is equivalent to
$$ \vert \hat{h}(\xi)\vert^{2} + \vert \hat{h}(\xi + \half) \vert^{2} \neq 0, \quad \xi \in \mathbb{T},$$

\noindent
which is automatically satisfied when $\vert \hat{h}(\xi)\vert >0$ for $\vert \xi \vert \le 1/4$.  It is now possible to describe sufficient conditions under which the infinitely iterated dyadic filter bank frame associated with low-pass filter $h$ and high-pass filter $g$ is stable. 

\begin{theorem} \label{expandingFB}
Let $h,g \in \ell^{2}(\mathbb{Z})$ with $h$ a low-pass filter and $g$ a high-pass filter.  Assume that $h$ is of the form \eqref{haar-type} with $p$ a trigonometric polynomial satisfying $p(0)=1$ and that $g$ is finitely supported.  Suppose there exists $s\in \mathbb{N}$ such that \eqref{p-est} holds and 
\begin{equation} \label{expanding}
M(\xi):=\frac{1}{\sqrt{2}} \begin{bmatrix} \hat{g}(\xi) & \hat{h}(\xi) \\ \hat{g}(\xi+\half) & \hat{h}(\xi+\half) \end{bmatrix}
\end{equation}

\noindent
is expanding in the sense that $\Vert M(\xi) c\Vert\ge \Vert c\Vert$ for all $c\in \mathbb{C}^{2}$ and almost every $\xi \in \mathbb{T}$.  Then, the infinitely iterated dyadic filter bank generated by $h$ and $g$ is stable.
\end{theorem}

\begin{proof}
The strategy for the proof is to show that the finitely iterated filter bank of order $j$ is a frame for $\ell^{2}(\mathbb{Z})$ with uniform bounds and then leverage Theorem \ref{FIFBtoIIFB}(b) to deduce the stability of the infinitely iterated dyadic filter bank.  Towards this end, fix $j\in \mathbb{N}$. Theorem \ref{BesselBound} implies that the infinitely iterated dyadic filter bank generated by $h$ and $g$ is Bessel with constant, say, $B_{1}>0$.  Meanwhile, Proposition \ref{xjshrinks} guarantees that there exists $B_{2}>0$ such that $\Vert (F_{j} x)_{j+1}\Vert^{2} \le B_{2} \Vert x\Vert^{2}$ for every $x\in \ell^{2}(\mathbb{Z})$ and all $j\in \mathbb{N}$.  It follows that
$$ \Vert (F_{j} x)_{j+1}\Vert^{2} + \sum_{l=1}^{j} \Vert (F_{j} x)_{l} \Vert^{2} \le (B_{1}+B_{2}) \Vert x\Vert^{2}$$

\noindent
for every $x\in \ell^{2}(\mathbb{Z})$, showing that the finitely iterated dyadic filter bank of order $j$ generated by $h$ and $g$ is Bessel with bound $B=B_{1}+B_{2}$.  Equivalently, this argument shows that $E_{2^{j}\mathbb{Z}}(\Phi_{j})$ is a Bessel system in $\ell^{2}(\mathbb{Z})$ with bound $B$, independent of $j$.

The fact that the $2\times 2$ matrix of \eqref{expanding} is expanding almost everywhere and its entries are trigonometric polynomials guarantees that it is non-singular for all $\xi \in \mathbb{T}$.  Consequently, \eqref{notzero} holds and Lemma \eqref{fullspan} guarantees that $\mathcal{S}_{2^{j}\mathbb{Z}}(\Phi_{j}) = \ell^{2}(\mathbb{Z})$.  Thus, the lower frame bound of $E_{2^{j}\mathbb{Z}}(\Phi_{j})$ with respect to $\ell^{2}(\mathbb{Z})$ can be determined by examining the Gramian, $\mathcal{G}_{\Phi_{j}}(\xi)$.  Recall that the set of generators $\Phi_{j}$ includes the iterated low-pass filter $h_{j}$ along with certain translates of the iterated high-pass filters $g_{1}, g_{2}, \ldots, g_{j}$.  More specifically, there are $2^{j-l}$ generators associated with the iterated high-pass filter $g_{l}$, namely, $\Psi_{l} := \lbrace  T^{2^{l}k} g_{l} : 0\le k\le 2^{j-2l} \rbrace$.  The pre-Gramian $\mathcal{T}_{\Phi_{j}}$ can thus be written in block form as
$$ \mathcal{T}_{\Phi_{j}} = \begin{bmatrix} \mathcal{T}_{\Psi_{1}} & \mathcal{T}_{\Psi_{2}} & \cdots & \mathcal{T}_{\Psi_{j}} & \mathcal{T} h_{j} \end{bmatrix},$$

\noindent
where $\mathcal{T}_{\Psi_{l}}$, $1\le l \le j$, is a block of size $2^{j} \times 2^{j-l}$.  Observe that $\mathcal{T}_{\Psi_{l}}(\xi)$ is equal to 
$$ \resizebox{\hsize}{!}{$
  2^{-\frac{j}{2}} \begin{bmatrix} 
\hat{g}_{l}(2^{-j}\xi)       & e^{-2\pi i 2^{l-j}\xi} \hat{g}_{l}(2^{-j}\xi) &  e^{-2\pi i 2^{l-j} 2\xi}  \hat{g}_{l}(2^{-j}\xi) & \cdots & e^{-2\pi i 2^{l-j}(2^{j-l}-1)\xi} \hat{g}_{l}(2^{-j}\xi) \\
\hat{g}_{l}(2^{-j}(\xi+1))       & e^{-2\pi i 2^{l-j}(\xi+1)} \hat{g}_{l}(2^{-j}(\xi+1)) &  e^{-2\pi i 2^{l-j} 2(\xi+1)}  \hat{g}_{l}(2^{-j}(\xi+1)) & \cdots & e^{-2\pi i 2^{l-j}(2^{j-l}-1)(\xi+1)} \hat{g}_{l}(2^{-j}(\xi+1)) \\
\vdots & \vdots & \vdots & \ddots  & \vdots \\
\hat{g}_{l}(2^{-j}(\xi+2^{j}-1)) & e^{-2\pi i 2^{l-j}(\xi+2^{j}-1)} \hat{g}_{l}(2^{-j}(\xi+2^{j}-1)) & e^{-2\pi i 2^{l-j}2(\xi+2^{j}-1)} \hat{g}_{l}(2^{-j}((\xi+2^{j}-1)) & \cdots & e^{-2\pi i 2^{-j}(2^{j-l}-1)(\xi+2^{j}-1)} \hat{g}_{l}(2^{-j}(\xi+2^{j}-1)) 
\end{bmatrix},$} $$

\noindent
which admits the factorization $\mathcal{T}_{\Psi_{l}}(\xi) = X_{\Psi_{l}}(\xi) \mathcal{F}_{j-l}(\xi)$, where $X_{\Psi_{l}}(\xi)$ is the $2^{j} \times 2^{j-l}$ matrix 
$$ \resizebox{0.75\hsize}{!}{$
 X_{\Psi_{l}}(\xi) = 2^{-\frac{l}{2}} \begin{bmatrix} \hat{g}_{l}(2^{-j}\xi) & 0 & \cdots & \cdots & 0 \\ 0 & \hat{g}_{l}(2^{-j}(\xi+1)) &  & & 0 \\ \vdots & 0 & \ddots & & \vdots \\ 0 & \vdots & & \ddots & 0 \\ 0 & 0 & & & \hat{g}_{l}(2^{-j}(\xi+2^{j-l}-1)) \\ \hat{g}_{l}(2^{-j}(\xi+ 2^{j-l})) & 0 & \cdots & \cdots & 0 \\ 0 & \hat{g}_{l}(2^{-j}(\xi+ 2^{j-l}+1)) & & & 0 \\ \vdots & 0 & \ddots & & \vdots \\ 0 & \vdots & & \ddots & 0 \\ 0 & 0 & & & \hat{g}_{l}(2^{-j}(\xi+2^{j+1-l}-1)) \\ \vdots & \vdots & & & \vdots \\ \hat{g}_{l}(2^{-j}(\xi+2^{j}-2^{j-l})) & 0 & \cdots & \cdots & 0 \\ 0 & \hat{g}_{l}(2^{-j}(\xi+2^{j}-2^{j-l}+1)) & & & 0 \\ \vdots & 0 & \ddots & & \vdots \\ 0 & \vdots & & \ddots & 0 \\ 0 & 0 & & & \hat{g}_{l}(2^{-j}(\xi+2^{j}-1)) \end{bmatrix}$} $$
 
\noindent
and $\mathcal{F}_{j-l}(\xi)$ is the $2^{j-l}\times 2^{j-l}$ Fourier matrix
$$ \mathcal{F}_{j-l}(\xi) = \begin{bmatrix} 1 & e^{-2\pi i 2^{l-j}\xi} & e^{-2\pi i 2^{l-j}2\xi} & \cdots & e^{-2\pi i 2^{l-j}(2^{j-l}-1)\xi} \\ 1 & e^{-2\pi i 2^{l-j}(\xi + 1)} & e^{-2\pi i 2^{l-j}2(\xi + 1)} & \cdots & e^{-2\pi i 2^{l-j}(2^{j-l}-1)(\xi + 1)} \\ \vdots & \vdots & \vdots & \ddots & \vdots  \\ 1 & e^{-2\pi i 2^{l-j} (\xi + 2^{j-l}-1)} & e^{-2\pi i 2^{l-j} 2 (\xi + 2^{j-l}-1)} & \cdots & e^{-2\pi i 2^{l-j}(2^{j-l}-1)(\xi + 2^{j-l}-1)} \end{bmatrix}. $$

\noindent
Notice that $X_{\Psi_{l}}$ is a vertical stack of $2^{l}$ diagonal $2^{j-l}\times 2^{j-l}$ matrices.  This factorization extends to fiberization of $\Phi_{j}$ as  $\mathcal{T}_{\Phi_{j}}(\xi) = X(\xi) \mathcal{F}(\xi)$, where 
$$ X = \begin{bmatrix} X_{\Psi_{1}} & X_{\Psi_{2}} & \cdots & X_{\Psi_{j}} & \mathcal{T} h_{j} \end{bmatrix}$$

\noindent
and $\mathcal{F}$ is the block diagonal matrix given by
$$ \mathcal{F} = \begin{bmatrix} \mathcal{F}_{j-1} & 0 & \cdots & 0 \\ 0 & \ddots & & \vdots \\ \vdots & & \mathcal{F}_{0} & 0 \\ 0 & \cdots & 0 & 1 \end{bmatrix}.$$

\noindent
The fact that $\mathcal{F}_{j-l}$ is unitary for $1\le l \le j$ implies that $\mathcal{F}$ is unitary and, therefore, the frame properties of $\mathcal{T}_{\Phi_{j}}(\xi)$ are identical to those of $X(\xi)$.

The final component of the proof involves a factorization of $X(\xi)$ as a product of expanding matrices, based on the product formula for the iterated low- and high-pass filters.  Let $\mathcal{H}_{l}(\xi)$, $1\le l \le j$, represent the $2^{j+1-l} \times 2^{j+1-l}$ matrix defined by
$$ \resizebox{0.95\hsize}{!}{$
\mathcal{H}_{l}(\xi) = \frac{1}{\sqrt{2}} \begin{bmatrix} \hat{g}(2^{l-j-1} \xi) & 0 & \cdots & 0 & \hat{h}(2^{l-j-1} \xi) & 0 & \cdots & 0 \\ 0 & \ddots & & \vdots & 0 & \ddots & & \vdots \\ \vdots & & \ddots & 0 & \vdots & & \ddots & 0 \\ 0 & \cdots & 0 & \hat{g}(2^{l-j-1} (\xi + 2^{j-l}-1)) & 0 & \cdots & 0 & \hat{h}(2^{l-j-1} (\xi + 2^{j-l}-1)) \\ \hat{g}(2^{l-j-1} (\xi+2^{j-l})) & 0 & \cdots & 0 & \hat{h}(2^{l-j-1} (\xi + 2^{j-l})) & 0 & \cdots & 0 \\ 0 & \ddots & & \vdots & 0 & \ddots & & \vdots \\ \vdots & & \ddots & 0 & \vdots & & \ddots & 0 \\ 0 & \cdots & 0 & \hat{g}(2^{l-j-1} (\xi + 2^{j+1-l}-1)) & 0 & \cdots & 0 & \hat{h}(2^{l-j-1} (\xi + 2^{j+1-l}-1)) \end{bmatrix}.$} $$

\noindent
It is important to note that $\mathcal{H}_{l}$, $1\le l \le j$, is expanding.  To see this, notice that it is possible to reorder the rows and columns so that $\mathcal{H}_{l}$ is a block-diagonal matrix composed of $2\times 2$ blocks of the form \eqref{expanding}.  Now, corresponding to each $\mathcal{H}_{l}$, define the $2^{j}\times 2^{j}$ matrix $Y_{l}$ by
$$ Y_{l}(\xi) = \begin{bmatrix} I_{2^{j}-2^{j+1-l}} & 0 \\ 0 & \mathcal{H}_{l}(\xi) \end{bmatrix},$$

\noindent
which must also be expanding for $1\le l \le j$.  The definitions of the iterated low- and high-pass filters imply that $X(\xi) = \prod_{l=1}^{j} Y_{l}(\xi)$.  Notice that the size of the identity block increases with $l$ as the product formula for the Fourier transforms of the iterated filters $g_{l}$ and $h_{l}$ contains exactly $l$ factors.  It follows that $X(\xi)$ is expanding for almost every $x\in \mathbb{T}$ as the product of $j$ such matrices.  Moreover, the fact that $\mathcal{F}$ is unitary implies that $\mathcal{T}_{\Phi_{j}}(\xi)$ is an expanding matrix for each $\xi \in \mathbb{T}$, ensuring that
$$ \langle \mathcal{G}_{\Phi_{j}}(\xi) c,c \rangle \ge \Vert c\Vert^{2}$$

\noindent
for every $\xi \in \mathbb{T}$ and every $c\in \mathbb{C}^{2^{j}}$.  Hence, by Proposition \ref{SIbounds}, $E_{2^{j}\mathbb{Z}}(\Phi_{j})$ is a frame for $\ell^{2}(\mathbb{Z})$ with lower bound $1$, independent of $j$.  Equivalently, the finitely iterated dyadic filter bank associate with $h$ and $g$ is stable, with uniform bounds.  Notice that Proposition \ref{xjshrinks} also guarantees that $\Vert (F_{j} x)_{j+1}\Vert^{2} \rightarrow 0$ as $j\rightarrow 0$ for all $x\in \ell^{2}(\mathbb{Z})$.  So, Theorem \ref{FIFBtoIIFB}(b) implies that the infinitely iterated dyadic filter bank generated by $h$ and $g$ is stable, completing the proof.
\end{proof}

Notice that when the orthogonal high-pass filter, given by \eqref{stdHP}, is chosen, the columns of
$$\frac{1}{\sqrt{2}} \begin{bmatrix} \hat{g}(\xi) & \hat{h}(\xi) \\ \hat{g}(\xi+\half) & \hat{h}(\xi+\half) \end{bmatrix}$$

\noindent
are orthogonal and the condition that this matrix is expanding is equivalent to
\begin{equation} \label{std-expand}
\vert \hat{h}(\xi)\vert^{2} + \vert \hat{h}(\xi+\half)\vert^{2} \ge 2,
\end{equation}

\noindent
further simplifying the sufficient conditions of Theorem \ref{expandingFB} for the stability of an applicable infinitely iterated dyadic filter bank.

Theorem \ref{expandingFB} addresses the question posed by Bayram and Selesnick by providing concrete conditions on the filters $h$ and $g$ that guarantee the stability of the infinitely iterated dyadic filter bank without referring to any associated scaling function or wavelet.  As noted by Bayram and Selesnick  \cite[Remark 9]{BayramSelesnick2009}, the relationship between the frame bounds of iterated filter banks and an underlying frame on the real line is delicate and not fully understood.

\section{Infinitely Iterated Filter Bank Frames} \label{examples}

This section is devoted to an informal illustration of the application of the results contained in this paper to the construction of low- and high-pass filters that give rise to dyadic filter banks which remain stable under an arbitrary number of iterations.

\subsection{Burt-Adelson Low-Pass Filters}

It was observed previously that the class of low-pass filters studied by Burt and Adelson \cite{BurtAdelson1983} was found to be compatible with the construction of biorthogonal wavelets given by Cohen, Daubechies, and Feauveau \cite{CohenDaubechiesFeauveau1992}.  The choice of high-pass filter used in the biorthogonal construction results from the solution of a specific B\'ezout equation.  Rather than covering the same ground here, the orthogonal high-pass filter, given by \eqref{stdHP}, will be considered.  The Burt-Adelson class of low-pass filters is typically described in terms of a parameter $a>0$, so that
\begin{equation} \label{BAfilter}
\frac{h(k)}{\sqrt{2}} = \begin{cases} 0.25-a/2, & k=\pm 2, \\ 0.25, & k=\pm 1, \\ a, & k=0, \\ 0, & \text{otherwise.} \end{cases}
\end{equation}

\noindent
This class of low-pass filters can also be written in the form \eqref{haar-type} using $n=2$ with
$$ p(\xi) = (4a-1) + (2-4a) \cos{(2\pi \xi)}.$$

\noindent
Observe that $p(\xi)$ is real-valued and its maximum on $\mathbb{T}$ (assuming $a>0.5$) is $8a-3$.  Therefore, \eqref{p-est} will hold with $s=1$ provided that $8a-3<2^{\frac{3}{2}}$, i.e., $a<(3+2\sqrt{2})/8$.  With $s=2$, it was verified numerically that \eqref{p-est} holds for $a\le 0.78$.  In order to apply Theorem \ref{expandingFB}, it must be verified that 
$$ \frac{1}{\sqrt{2}} \begin{bmatrix} \hat{g}(\xi) & \hat{h}(\xi) \\ \hat{g}(\xi+\half) & \hat{h}(\xi+\half) \end{bmatrix}$$

\noindent
is expanding on $\mathbb{T}$; however, in light of the remark following Theorem \ref{expandingFB}, the stability of the infinitely iterated dyadic filter bank is guaranteed when $\vert \hat{h}(\xi)\vert^{2} + \vert \hat{h}(\xi + \half)\vert^{2} \ge 2$ on $\mathbb{T}$.  Figure \ref{BAgraph} depicts numerical estimates of this quantity for certain values of the parameter $a$, equally spaced over the range $0.5\le a \le 0.78$.  Based on these computations, one can reasonably conclude that \eqref{std-expand} holds when $a$ is greater than about 0.625.  Theorem \ref{expandingFB} then implies that the infinitely iterated dyadic filter bank is stable in the range $0.625 \le a \le 0.78$, while Theorem \ref{IIFBtoFIFB} implies that the associated finitely iterated filter banks will be stable over the same range of parameters and with uniform bounds.  Since the lower frame bound established by Theorem \ref{expandingFB} is always 1, one should expect that the smaller the choice of $a$, the tighter the resulting frame will be.

\begin{figure}[hbtp]
\centering
\includegraphics[width=6.0in]{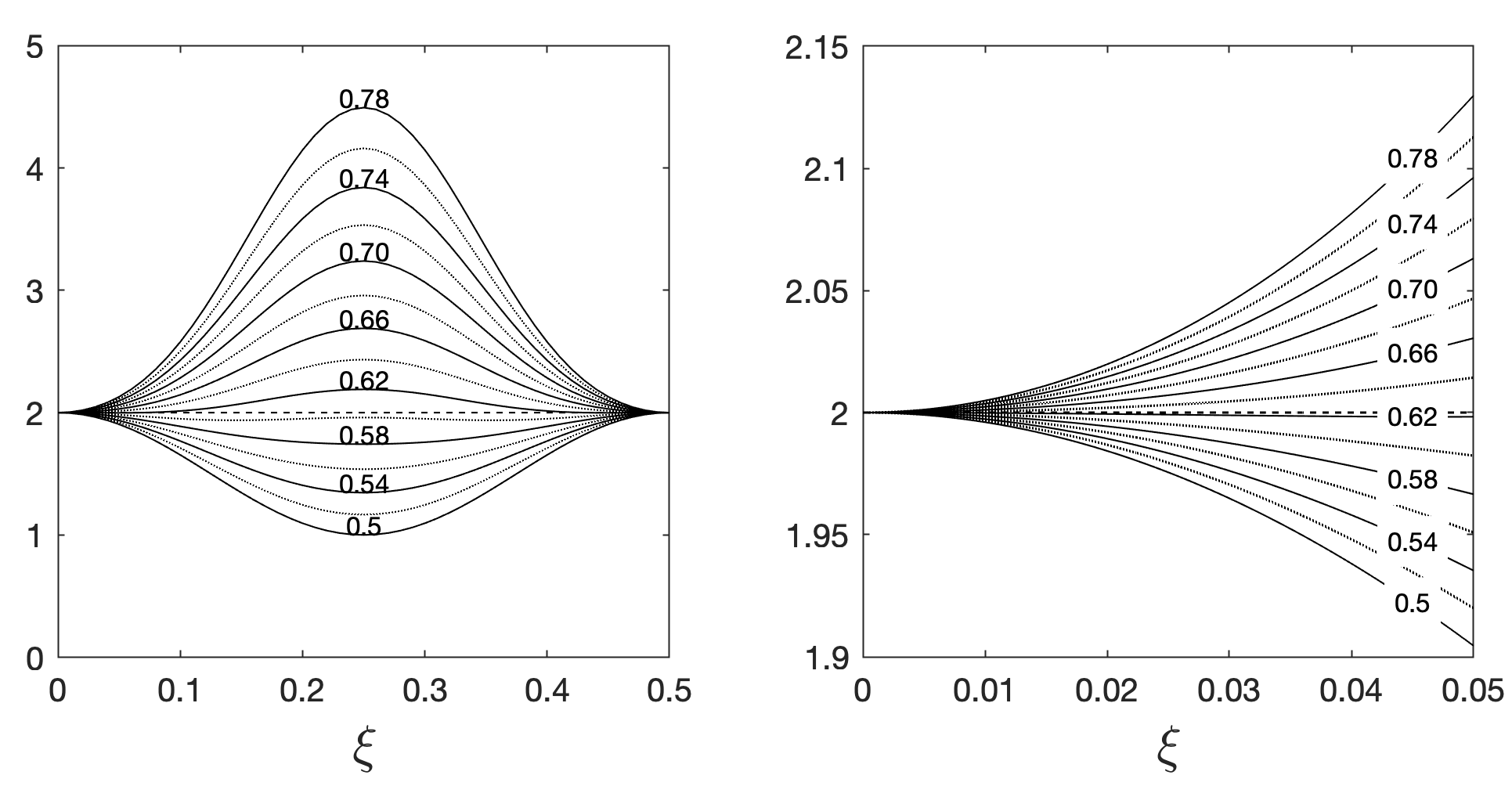}

\caption{The quantity $\vert \hat{h}(\xi)\vert^{2} + \vert \hat{h}(\xi+\half)\vert^{2}$ as a function of the parameter $a$ for the Burt-Adelson class of low-pass filters.} \label{BAgraph}
\end{figure}

\subsection{A Higher-Order Class of Low-Pass Filters} \label{higherorder}

As a second illustration of the simplicity of the sufficient conditions described by Theorem \ref{expandingFB}, consider the low-pass filter of the form \eqref{haar-type} with $n=3$ and $p(\xi)$ given by
$$ p(\xi) = (1+2a) -2a \cos{(2\pi \xi)}.$$

\noindent
The maximum value of $p(\xi)$ is $1+4a$ (assuming $a>0$) and occurs when $x=\pm \half$.  Therefore, in order to satisfy \eqref{p-est} with $s=1$, it follows that $1+4a<4\sqrt{2}$ or $a<\sqrt{2}-\frac{1}{4}$.  With $s=2$, it was verified numerically that \eqref{p-est} holds for $a\le 1.5$  The high-pass filter will again be chosen according to \eqref{stdHP} and the parameter $a$ will be considered over the range $0\le a\le 1.5$.  Figure \ref{MZgraph} depicts numerical estimates of $\vert \hat{h}(\xi)\vert^{2} + \vert \hat{h}(\xi+\half)\vert^{2}$ for equally spaced values of the parameter $a$, which suggests that the expanding condition is satisfied for $a$ larger than 0.5.  As above, Theorem \ref{expandingFB} then guarantees the stability of the infinitely iterated dyadic filter bank and Theorem \ref{IIFBtoFIFB} then implies the stability of the associated finitely iterated dyadic filter banks with uniform bounds.

\begin{figure}[hbtp]
\centering
\includegraphics[width=6.0in]{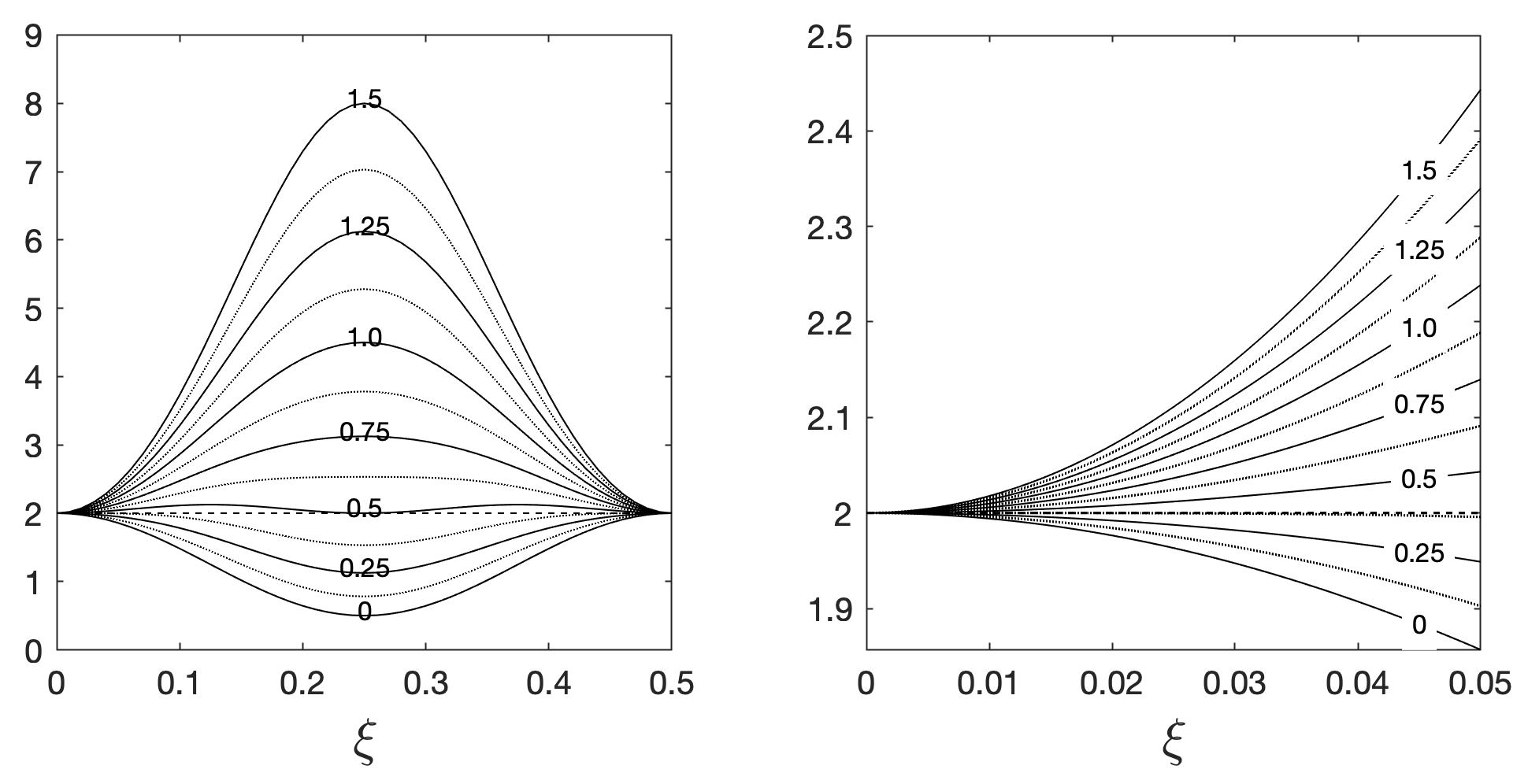}

\caption{The quantity $\vert \hat{h}(\xi)\vert^{2} + \vert \hat{h}(\xi+\half)\vert^{2}$ as a function of the parameter $a$ for the class of low-pass filters of Section \ref{higherorder}.} \label{MZgraph}
\end{figure}

\section{Conclusion}

The primary goal of this paper is to address a question raised by Bayram and Selesnick in \cite{BayramSelesnick2009} by providing straightforward conditions on the low- and high-pass filters that guarantee the stability of the associated dyadic filter bank under an arbitrary number of iterations.  Theorem \ref{BesselBound} establishes a Bessel bound for infinitely iterated dyadic filter banks associated with a broad class of finitely supported filters, while Theorem \ref{expandingFB} yields a lower frame bound using a fiberization approach common to the study of shift-invariant spaces.  The argument of Theorem \ref{expandingFB} relies on the fact that the matrix given by \eqref{expanding} is expanding almost everywhere, ensuring that the lower frame bound equals 1.

These results were used to examine a class of iterated filter banks employing a Burt-Adelson low-pass filter together with the orthogonal high-pass filter given by \eqref{stdHP}.  It was observed that such filter banks are stable for a wide range of the parameter $a$ and the authors know of no other methods to easily establish the stability of these filter banks.  Still, it was not possible to establish stability for the widely used parameter value $a=0.6$ due to the fact that \eqref{expanding} is not expanding on some intervals.  In fact, with $a=0.6$, the expanding condition also fails when the orthogonal high-pass filter associated with \eqref{BAfilter} is replaced by the biorthogonal high-pass filter 
$$ \hat{g}(\xi) = e^{-2\pi i \xi} \hat{\tilde{h}}(\xi + \half),$$

\noindent
where $\tilde{h}$ represents the biorthogonal dual low-pass filter \cite[Table 6.3]{CohenDaubechiesFeauveau1992}.  This can be seen by examining the eigenvalue functions of 
\begin{equation} \label{notexpanding}
\frac{1}{2} \begin{bmatrix} \hat{g}(\xi) & \hat{h}(\xi) \\ \hat{g}(\xi+\half) & \hat{h}(\xi+\half) \end{bmatrix}^{*} \begin{bmatrix} \hat{g}(\xi) & \hat{h}(\xi) \\ \hat{g}(\xi+\half) & \hat{h}(\xi+\half) \end{bmatrix},
\end{equation}

\noindent
which are shown in Figure \ref{BAsixgraph}.  Nevertheless, despite the failure of the pointwise expanding property for this filter pair, it follows from the work of Bayram and Selesnick that the finitely iterated dyadic filter banks associated with these filters remain stable under arbitrarily many iterations.

\begin{figure}[hbtp]
\centering
\includegraphics[width=5.0in]{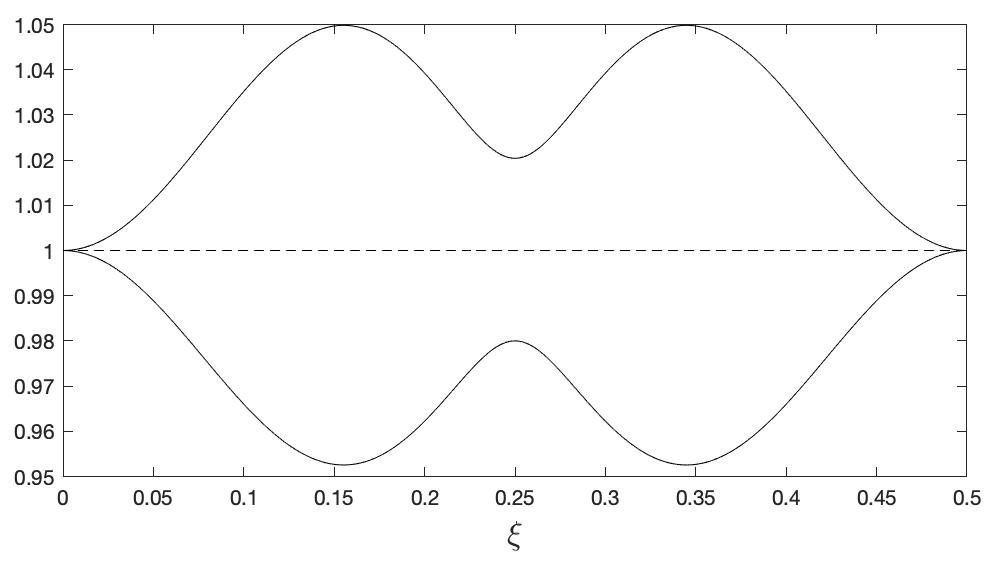}

\caption{The eigenvalue functions of \eqref{notexpanding} for the filters associated with the Burt-Adelson biorthogonal wavelet and $a=0.6$.} \label{BAsixgraph}
\end{figure}

The sufficient conditions presented here are limited to filter banks for which \eqref{expanding} is expanding almost everywhere.  In light of this limitation, there is room for further study of the stability of iterated dyadic filter banks in which the \emph{pointwise} expanding condition used in Theorem \ref{expandingFB} might be replaced by an \emph{average} expanding condition derived with the help of tools from ergodic theory, as found in the works of Conze and Raugi \cite{ConzeRaugi1990} or Gundy \cite{Gundy2000}.  It may also be possible to use perturbation methods \cite[Corollary 22.1.5]{Christensen2016} to broaden the class of filters that give rise to iterated filter banks which remain stable under an arbitrary number of iterations.

\end{document}